\documentclass[11pt]{article}

\usepackage[a4paper,margin=1in]{geometry}
\usepackage{amsmath,amssymb,amsthm,mathtools}
\usepackage{booktabs}
\usepackage{array,tabularx}
\usepackage{graphicx}
\usepackage{float}
\usepackage{microtype}
\usepackage{cite}
\usepackage[hidelinks]{hyperref}
\hypersetup{
  pdftitle={Making Surfaces Biharmonic by Conformal Reparametrization in Anti-de Sitter Three-Space},
  pdfauthor={Dipesh Bhandari},
  pdfsubject={Biharmonic conformal immersions into anti-de Sitter three-space},
  pdfkeywords={biharmonic map, conformal immersion, anti-de Sitter space, pseudo-Riemannian submanifold, rotational surface, f-biharmonic map}
}

\numberwithin{equation}{section}

\newtheorem{theorem}{Theorem}[section]
\newtheorem{proposition}[theorem]{Proposition}
\newtheorem{lemma}[theorem]{Lemma}
\newtheorem{corollary}[theorem]{Corollary}
\theoremstyle{definition}
\newtheorem{example}[theorem]{Example}
\theoremstyle{remark}
\newtheorem{remark}[theorem]{Remark}

\newcommand{\AdS}{\operatorname{AdS}}
\newcommand{\grad}{\operatorname{grad}}
\newcommand{\tr}{\operatorname{tr}}

\newcommand{\diver}{\operatorname{div}}
\newcolumntype{Y}{>{\raggedright\arraybackslash}X}

\title{\textbf{Making Surfaces Biharmonic by Conformal Reparametrization in Anti-de Sitter Three-Space}\\[0.35em]
\large Rigidity, Local Existence, and Parabolic Rotational Families}

\author{Dipesh Bhandari\\
\small Department of Physics, Southern Methodist University, Dallas, Texas 75275, USA\\
\small \texttt{dbhandari@smu.edu}}
\date{}
\begin{document}

\maketitle

\begin{abstract}
Harmonic immersions of surfaces are minimal, while biharmonic maps form a
fourth-order extension of harmonic-map theory.  Because every harmonic map
is automatically biharmonic, the basic existence problem is to find
\emph{proper} biharmonic maps, namely biharmonic maps that are not harmonic.
This paper asks a more geometric question: when can a fixed nondegenerate
surface in three-dimensional anti-de Sitter space be made proper biharmonic
by changing only the conformal metric on its domain?  Equivalently, how much
of biharmonicity is determined by the immersed surface, and how much can be
created by conformal reparametrization?

Writing the induced metric as $g=\lambda^2\bar g$ and introducing the
weighted mean curvature $u=\lambda^2H$, we first reduce the map equation to
a normal scalar equation coupled to a tangential first-order constraint.
The resulting system reveals a sharp rigidity--existence dichotomy.  A
nonminimal spacelike constant-mean-curvature solution must have constant
dilation and is locally the totally umbilical hyperbolic plane of curvature
$-2/L^2$.  Once the constant-mean-curvature assumption is removed, however,
there is an open set of local analytic solutions for which both $H$ and
$\lambda$ vary.  A moving-frame invariant then identifies the ambient
one-parameter symmetry and separates elliptic, hyperbolic, and index-three
parabolic orbit types.  In the parabolic class the geometric system reduces
to a scalar third-order analytic equation, from which we reconstruct
explicit local spacelike and real-principal timelike families in null
coordinates.  The paper therefore locates the rigid branch, proves that the
rigidity can be escaped, and gives an explicit mechanism for producing the
resulting non-CMC surfaces.
\end{abstract}

\noindent\textbf{Keywords.}
Biharmonic map; conformal immersion; anti-de Sitter space;
pseudo-Riemannian submanifold; rotational surface; $f$-biharmonic map.

\medskip
\noindent\textbf{2020 Mathematics Subject Classification.}
58E20, 53C42, 53C50, 53B30.

\section{Introduction}
\label{sec:introduction}

\subsection*{From harmonic maps to the geometric problem studied here}

Harmonic maps are the natural energy-critical maps between manifolds.  If
$\phi:(M,g)\to(N,h)$ is smooth, its Dirichlet energy measures the size of
$d\phi$, and its Euler--Lagrange equation is the vanishing of the tension
field $\tau(\phi)$.  This single framework contains geodesics, harmonic
functions, and many of the basic variational objects of differential
geometry.  For an isometric immersion of a surface, the tension field is
twice the mean-curvature vector.  Thus harmonicity is exactly minimality.
Harmonic-map theory also provides geometric models used in elasticity,
field theory, and the analysis of constrained media; see
\cite{Helein2002,Branding2024} for background and perspective.

Biharmonic maps are critical points of the bienergy
\[
    E_2(\phi)=\frac12\int_M |\tau(\phi)|^2\,dv_g
\]
and satisfy the fourth-order equation introduced by Jiang
\cite{Jiang1986}.  They are a higher-order analogue of harmonic maps, but
there is an immediate degeneracy: every harmonic map is biharmonic.  The
interesting solutions are therefore the \emph{proper biharmonic maps}, for
which $\tau(\phi)\neq0$ but the bitension field vanishes.  In submanifold
geometry this leads to a concrete question: which nonminimal immersed
submanifolds satisfy the fourth-order balance law?  Surveys of this theory
include \cite{MontaldoOniciuc2006,Ou2019}.

For surfaces there is a second degree of freedom that is easy to overlook.
An immersion determines its induced metric $g=\phi^*h$, but the same map may
also be viewed from a conformally related source metric
\[
    \bar g=\lambda^{-2}g,
    \qquad \lambda>0.
\]
Changing $\lambda$ does not change the image surface or its angles; it
changes the way the domain measures lengths and hence changes the tension
and bitension fields of the map.  This separates two geometric ingredients:
the extrinsic shape of the surface and the conformal parametrization used to
probe it.  In dimension two, the problem is equivalent to the
$f$-biharmonicity of the associated isometric immersion, with
$f=\lambda^2$ \cite{Ou2009,Ou2014,Ou2015,Ou2017}.

The present paper studies this separation in anti-de Sitter three-space,
the basic Lorentzian space form of constant negative curvature.  This is a
particularly revealing setting.  In Riemannian negative curvature, many
biharmonic theories are dominated by nonexistence and rigidity.  In
$\AdS_3$, the sign of the unit normal changes the shape-operator term in the
normal equation, so one should not expect the Riemannian conclusions to
carry over unchanged.  The spacelike and timelike cases also test how a
fourth-order variational equation interacts with causal geometry.  The
paper is geometric rather than model-specific, but it belongs to the same
higher-order variational framework that motivates biharmonic maps in
analysis, elasticity, and mathematical physics.

\medskip
\noindent\textbf{Guiding problem.}
Given a nondegenerate immersed surface
\[
    \phi:M^2\longrightarrow\AdS_3(L)
\]
with induced metric $g$, determine when there is a positive function
$\lambda$ such that
\[
    \phi:(M^2,\lambda^{-2}g)\longrightarrow\AdS_3(L)
\]
is proper biharmonic.  We organize this problem into three questions:

\begin{enumerate}
\item \emph{Rigidity:} if the surface has constant mean curvature, can the
      conformal dilation vary?
\item \emph{Existence beyond rigidity:} if constant mean curvature is
      dropped, do nonconstant dilations occur robustly, or only in isolated
      examples?
\item \emph{Extrinsic realization:} when symmetry reduces the equations to
      ordinary differential equations, which ambient one-parameter group
      actually generates the surface, and can the immersion be reconstructed
      explicitly?
\end{enumerate}

These questions give the paper its progression: first identify the rigid
branch, then show how to leave it, and finally recover the actual ambient
geometry of the new solutions.

\subsection*{The main answers and why they are of interest}

The key variable is not $H$ or $\lambda$ separately, but the weighted mean
curvature
\[
    u=\lambda^2H.
\]
It combines the extrinsic bending of the surface with the conformal weight
of the source metric.  In Section~\ref{sec:conformal-biharmonic-AdS} we show
that the fourth-order map equation is equivalent to
\[
    \Delta_gu+
    \left(-\varepsilon\operatorname{tr}_g(A^2)+2c\right)u=0,
    \qquad
    A(\grad_gu)+\varepsilon u\grad_gH=0.
\]
The first equation is a weighted normal balance; the second couples the
gradient of the weight to the principal geometry.  This formulation is the
working bridge between biharmonic-map theory and surface geometry.

The first answer is a rigidity theorem.  For a two-sided spacelike surface
with nonzero constant mean curvature, the tangential equation leaves no room
for a genuinely varying conformal factor.  We prove that $\lambda$ is
constant and that the surface is locally totally umbilical, with
\[
    A=\pm L^{-1}\operatorname{id},
    \qquad K=-2L^{-2}.
\]
Thus conformal reparametrization does not enlarge the nonminimal CMC branch:
it only produces constant homotheties of the known proper biharmonic
isometric surface.  This is useful not merely as a classification result; it
identifies exactly where a search for new examples must fail.

The second answer is that the failure is specific to CMC geometry.  Using a
warped principal-coordinate ansatz, we derive a closed analytic ODE system
whose initial data determine the metric, shape operator, weighted mean
curvature, and dilation.  An open set of initial data gives
$dH\neq0$ and $d\lambda\neq0$.  Hence the non-CMC solutions are not isolated
formal examples: they persist under small perturbations of the initial
profile.  This establishes a genuine local moduli phenomenon on the other
side of the CMC rigidity theorem.

The third answer explains what those local data look like in the ambient
space.  An ambient moving-frame calculation produces the conserved quantity
\[
    \mathfrak c=\eta\rho^2
    \left(L^{-2}-\sigma q^2-\varepsilon k_2^2\right)
\]
and a constant rank-two generator $\mathcal B\in\mathfrak{so}(2,2)$ satisfying
$\mathcal B^3=\mathfrak c\mathcal B$.  Its sign separates elliptic,
hyperbolic, and index-three parabolic orbit geometry.  On the parabolic
branch the full surface problem reduces to one scalar third-order analytic
ODE.  The scalar solution then determines the immersion by quadratures in
null coordinates.  In this way the reduction is not left at the level of
abstract first and second fundamental forms: the actual surface and its
ambient symmetry are recovered.

Table~\ref{tab:problem-answer} summarizes the logical structure.

\begin{table}[htbp]
\centering
\small
\begin{tabularx}{\textwidth}{@{}p{0.22\textwidth}Y Y@{}}
\toprule
Question & Answer proved here & Geometric meaning \\
\midrule
Can a nonzero-CMC spacelike surface acquire biharmonicity from a variable
conformal factor?
& No: the dilation is constant and the surface is locally totally
umbilical.
& The CMC branch is rigid; no new conformal freedom is created there. \\
Do variable-dilation solutions exist once CMC is removed?
& Yes: an open set of analytic cohomogeneity-one initial data gives
nonconstant $H$ and $\lambda$.
& The rigidity is not a general nonexistence theorem; it marks a boundary
between two regimes. \\
Can the reduced solutions be identified as actual rotational surfaces?
& Yes: a conserved generator classifies the orbit, and the parabolic
branches admit explicit null-coordinate reconstruction.
& The intrinsic ODE data are connected to concrete ambient geometry. \\
\bottomrule
\end{tabularx}
\caption{The problem--answer structure of the paper.}
\label{tab:problem-answer}
\end{table}

\subsection*{Relation to earlier work and precise scope}

The general pseudo-Riemannian biharmonic submanifold equations were derived
by Dong and Ou \cite{DongOu2017}.  The isometric theory in Lorentzian space
forms includes classifications of proper biharmonic spacelike surfaces and
hypersurfaces in anti-de Sitter space
\cite{Zhang2008,Sasahara2012,LiuDu2015}, classifications of
$\eta$-biharmonic Lorentzian surfaces \cite{Du2018}, and higher-order
polyharmonic hypersurfaces in pseudo-Riemannian space forms
\cite{BrandingEtAl2023}.  Lorentzian $f$-biharmonic submanifolds have also
been studied under assumptions such as parallel normalized mean curvature
and pseudo-umbilicity \cite{Du2021}.  Those hypotheses do not include the
non-CMC warped principal-coordinate families constructed here.

Recent Riemannian work treats biharmonic conformal immersions into
three-dimensional conformally flat spaces \cite{WangChen2026}, conformal
hypersurfaces in general Riemannian manifolds \cite{CherifOu2026}, and
$f$-biharmonic hypersurfaces in space forms and conformally flat spaces
\cite{WangQin2023,WangQinChen2024}.  Rotational hypersurfaces and their
elliptic, hyperbolic, and parabolic orbit types are classical
\cite{doCarmoDajczer1983}.  Related rotational families in Lorentzian space
forms occur in the classification of biconservative surfaces
\cite{Fu2015}; that problem imposes only the tangential bitension equation,
whereas the present problem imposes both the weighted normal and tangential
equations and reconstructs the conformal dilation.  The rotational orbit
geometry itself is therefore not claimed as new.

A separate literature uses ``conformal biharmonic'' for critical points of a
curvature-corrected conformal bienergy \cite{BrandingNistorOniciuc}.  That is
not the variational problem considered here.  Throughout this paper,
``biharmonic conformal immersion'' means a conformal immersion that is
biharmonic as a map with respect to its stated source metric, and ``proper
biharmonic'' means biharmonic but nonharmonic, not proper as a topological
map.

The signed surface equation in Section~\ref{sec:conformal-biharmonic-AdS} is
included to make the conventions and reduction self-contained; it is not by
itself the main novelty.  The new content is the CMC rigidity theorem, the
open non-CMC local existence theory, the conserved ambient-generator
classification, and the explicit parabolic reconstructions.  All existence
statements are local.  We do not claim a global classification, completeness
of the new surfaces, periodic elliptic or hyperbolic profiles, or properness
of the immersion as a map.

\begin{table}[htbp]
\centering
\small
\begin{tabularx}{\textwidth}{@{}p{0.20\textwidth}Y Y Y@{}}
\toprule
Theory & Typical hypothesis & Equation imposed & Relation to this paper \\
\midrule
Isometric biharmonic Lorentzian surfaces
\cite{Sasahara2012,LiuDu2015}
& Fixed induced domain metric, often CMC or restricted principal-curvature
type
& Full unweighted bitension equation
& The constant-dilation subcase of the present problem. \\
$\eta$-biharmonic surfaces \cite{Du2018}
& Linear modification of the bitension field
& Modified fourth-order equation
& A different variational equation. \\
Lorentzian $f$-biharmonic submanifolds \cite{Du2021}
& Parallel normalized mean curvature, pseudo-umbilicity, or related
restrictions
& Weighted biharmonic equation
& Does not cover the non-CMC warped families constructed below. \\
Biconservative rotational surfaces \cite{Fu2015}
& Vanishing tangential bitension component
& Tangential equation only
& Shares orbit geometry, but not the weighted normal equation or recovered
dilation. \\
Present work
& Local spacelike data and one specified real-principal timelike branch
& Both weighted normal and tangential equations
& Produces local families with both $H$ and $\lambda$ nonconstant. \\
\bottomrule
\end{tabularx}
\caption{Position of the present problem relative to adjacent surface
theories.}
\label{tab:scope-comparison}
\end{table}

\subsection*{Roadmap}

Section~\ref{sec:conformal-biharmonic-AdS} translates the biharmonic map
equation into the two surface equations used throughout the paper.
Section~3 applies the tangential equation to the CMC branch and proves its
rigidity.  Section~\ref{sec:cohomogeneity-one-existence} then removes the CMC
assumption and constructs an open local family from analytic initial data.
Section~\ref{sec:analytical-rotational-reduction} identifies the ambient
symmetry generator, classifies its orbit type, and develops the spacelike and
timelike parabolic reconstructions.  The concluding section returns to the
guiding problem, separates what has been settled locally from what remains
global, and states the next natural questions.

\section{From the biharmonic map equation to a surface system}
\label{sec:conformal-biharmonic-AdS}

The guiding problem is stated in terms of a fourth-order map equation, but
the later rigidity and existence arguments require quantities that belong
directly to the immersed surface.  The purpose of this section is therefore
to translate biharmonicity into a coupled system for the induced metric, the
shape operator, the mean curvature, and the conformal dilation.  Keeping the
normal sign and pseudo-Riemannian traces explicit lets the same derivation
cover both spacelike and timelike surfaces.

The reader should keep the roles of the two metrics distinct.  The induced
metric $g=\phi^*h$ records the geometry of the image surface, while
$\bar g=\lambda^{-2}g$ is the metric with respect to which the map is asked
to be biharmonic.  The weighted variable $u=\lambda^2H$ is where those two
geometries meet.

\subsection{The two metrics and the causal sign conventions}

Let $(N^3,h)$ be a Lorentzian space form of constant sectional curvature $c$,
with curvature tensor
\begin{equation}
R^N(X,Y)Z
=
 c\bigl(h(Y,Z)X-h(X,Z)Y\bigr).
\label{eq:ambient-curvature-convention}
\end{equation}
For anti-de Sitter three-space of radius $L>0$,
\begin{equation}
N^3=\AdS_3(L),
\qquad
c=-\frac{1}{L^2}.
\label{eq:ads-curvature}
\end{equation}

Let
\begin{equation}
\phi:(M^2,\bar g)\longrightarrow (N^3,h)
\label{eq:conformal-immersion}
\end{equation}
be a nondegenerate conformal immersion, and write
\begin{equation}
g=\phi^*h=\lambda^2\bar g,
\qquad
\bar g=\lambda^{-2}g,
\qquad
\lambda>0.
\label{eq:metric-relation}
\end{equation}
Thus $g$ is the induced metric.  The immersion is spacelike when $g$ is
Riemannian and timelike when $g$ is Lorentzian.

Choose a local unit normal field $\xi$ and define
\begin{equation}
\varepsilon:=h(\xi,\xi)
=
\begin{cases}
-1, & \text{for a spacelike surface},\\
+1, & \text{for a timelike surface}.
\end{cases}
\label{eq:epsilon-definition}
\end{equation}
Our Weingarten convention is
\begin{equation}
\nabla^N_X\xi=-AX,
\label{eq:weingarten}
\end{equation}
so the Gauss formula is
\begin{equation}
\nabla^N_XY
=
\nabla_XY+\varepsilon g(AX,Y)\xi.
\label{eq:gauss-formula}
\end{equation}

We use the standard semi-Riemannian submanifold conventions of
Ref.~\cite{ONeill1983}, subject to the explicit curvature and
mean-curvature sign choices stated below.

We define the scalar mean curvature $H$ by requiring the mean-curvature vector
to be
\begin{equation}
\mathbf H=H\xi.
\label{eq:mean-curvature-vector}
\end{equation}
Consequently,
\begin{equation}
H=\frac{\varepsilon}{2}\tr_g A,
\qquad
\tr_g A=2\varepsilon H.
\label{eq:H-trA-relation}
\end{equation}
In particular,
\begin{equation}
H=-\frac12\tr_gA
\quad\text{in the spacelike case},
\qquad
H=\frac12\tr_gA
\quad\text{in the timelike case}.
\label{eq:H-special-cases}
\end{equation}
This convention is convenient because the tension field of the induced
isometric immersion is simply
\begin{equation}
\tau_g(\phi)=2H\xi.
\label{eq:tension-isometric}
\end{equation}

Let $\{e_1,e_2\}$ be a local pseudo-orthonormal frame for $g$, with
\begin{equation}
g(e_i,e_j)=\epsilon_i\delta_{ij},
\qquad
\epsilon_i\in\{+1,-1\}.
\end{equation}
For a function $q$ and a section $V$ of $\phi^{-1}TN$, we use
\begin{align}
\Delta_g q
&=
\sum_{i=1}^2\epsilon_i
\left(e_i(e_iq)-(\nabla_{e_i}e_i)q\right),
\label{eq:scalar-laplacian}\\
\Delta_g^\phi V
&=
\sum_{i=1}^2\epsilon_i
\left(
\nabla^\phi_{e_i}\nabla^\phi_{e_i}V
-
\nabla^\phi_{\nabla_{e_i}e_i}V
\right).
\label{eq:rough-laplacian}
\end{align}
Thus $\Delta_g$ is the Laplace--Beltrami operator in the spacelike case and
the wave operator in the timelike case.  We also set
\begin{equation}
|A|_g^2:=\tr_g(A^2)
=
\sum_{i=1}^2\epsilon_i g(Ae_i,Ae_i).
\label{eq:A-square}
\end{equation}
For a timelike surface this quantity need not be positive.

The bitension field is defined by
\begin{equation}
\tau_{2,g}(\phi)
=
\Delta_g^\phi\tau_g(\phi)
-
\sum_{i=1}^2\epsilon_i
R^N\bigl(d\phi(e_i),\tau_g(\phi)\bigr)d\phi(e_i).
\label{eq:bitension-definition}
\end{equation}
For timelike domains, ``biharmonic'' is understood in the standard
pseudo-Riemannian sense: the field in \eqref{eq:bitension-definition} vanishes.
The pseudo-Riemannian trace convention agrees with the one used in
\cite{DongOu2017}.

\subsection{Why conformal reparametrization becomes an \texorpdfstring{$f$}{f}-biharmonic problem}

Set
\begin{equation}
f:=\lambda^2,
\qquad
\bar g=f^{-1}g.
\label{eq:f-lambda}
\end{equation}
The following two-dimensional conformal-change identity is the
pseudo-Riemannian counterpart of the standard Riemannian formula in
\cite{Ou2009}.

\begin{lemma}[Two-dimensional conformal reduction]
\label{lem:conformal-reduction}
Let $\phi:(M^2,g)\to(N,h)$ be a map between pseudo-Riemannian manifolds, and
let $\bar g=f^{-1}g$ for a positive function $f$.  Then
\begin{equation}
\tau_{\bar g}(\phi)=f\tau_g(\phi),
\label{eq:tension-conformal-transform}
\end{equation}
and
\begin{equation}
\tau_{2,\bar g}(\phi)
=
f\left[
 f\tau_{2,g}(\phi)
 +(\Delta_g f)\tau_g(\phi)
 +2\nabla^\phi_{\grad_g f}\tau_g(\phi)
\right].
\label{eq:bitension-conformal-transform}
\end{equation}
Consequently, $\phi:(M^2,f^{-1}g)\to(N,h)$ is biharmonic if and only if
\begin{equation}
f\tau_{2,g}(\phi)
 +(\Delta_g f)\tau_g(\phi)
 +2\nabla^\phi_{\grad_g f}\tau_g(\phi)=0.
\label{eq:f-biharmonic-map-equation}
\end{equation}
\end{lemma}

\begin{proof}
For a conformal change $\bar g=f^{-1}g$ in dimension two, the trace term
arising from the difference of the Levi--Civita connections cancels, giving
\eqref{eq:tension-conformal-transform}.  Introduce the linear operator
\begin{equation}
\mathcal L_g^\phi(V)
:=
\Delta_g^\phi V
-
\sum_{i=1}^2\epsilon_i
R^N\bigl(d\phi(e_i),V\bigr)d\phi(e_i),
\label{eq:L-operator}
\end{equation}
so that $\tau_{2,g}(\phi)=\mathcal L_g^\phi(\tau_g(\phi))$.
If $\{e_i\}$ is pseudo-orthonormal for $g$, then
$\{\sqrt f\,e_i\}$ is pseudo-orthonormal for $\bar g$.  The standard
connection-difference formula shows that the first-order trace term in the
rough Laplacian is proportional to $m-2$ and hence vanishes for $m=2$.
The curvature trace scales by the same factor $f$.  Consequently,
\begin{equation}
\mathcal L_{\bar g}^\phi(V)=f\mathcal L_g^\phi(V).
\label{eq:L-conformal-transform}
\end{equation}
Using the product rule
\begin{equation}
\mathcal L_g^\phi(fV)
=
f\mathcal L_g^\phi(V)
 +(\Delta_g f)V
 +2\nabla^\phi_{\grad_g f}V,
\label{eq:L-product-rule}
\end{equation}
we obtain
\begin{align}
\tau_{2,\bar g}(\phi)
&=
\mathcal L_{\bar g}^\phi\bigl(\tau_{\bar g}(\phi)\bigr)\\
&=
f\mathcal L_g^\phi\bigl(f\tau_g(\phi)\bigr)\\
&=
f\left[
 f\tau_{2,g}(\phi)
 +(\Delta_g f)\tau_g(\phi)
 +2\nabla^\phi_{\grad_g f}\tau_g(\phi)
\right],
\end{align}
which proves the claim.  The computation uses only the signed trace and
therefore applies equally to spacelike and timelike metrics.
\end{proof}

\subsection{The normal and tangential surface equations}

\begin{lemma}[Bitension operator on the mean-curvature normal]
\label{lem:normal-bitension}
For the induced isometric immersion $\phi:(M^2,g)\to N^3(c)$,
\begin{align}
\mathcal L_g^\phi(H\xi)
={}&
\left(
\Delta_gH
-\varepsilon |A|_g^2H
+2cH
\right)\xi
\nonumber\\
&
-2A(\grad_gH)
-2\varepsilon H\grad_gH.
\label{eq:L-Hxi}
\end{align}
\end{lemma}

\begin{proof}
Fix a point $p\in M$ and choose a local pseudo-orthonormal frame satisfying
$\nabla_{e_i}e_j=0$ at $p$.  From the Weingarten formula,
\begin{equation}
\nabla^\phi_{e_i}(H\xi)
=e_i(H)\xi-HAe_i.
\label{eq:first-Hxi-derivative}
\end{equation}
Differentiating once more, tracing with the signs $\epsilon_i$, and using the
Gauss formula gives
\begin{equation}
\Delta_g^\phi(H\xi)
=
\left(\Delta_gH-\varepsilon H|A|_g^2\right)\xi
-2A(\grad_gH)-H\diver_gA,
\label{eq:rough-Hxi}
\end{equation}
where
\begin{equation}
\diver_gA
:=
\sum_{i=1}^2\epsilon_i(\nabla_{e_i}A)e_i.
\end{equation}
Since the ambient space has constant sectional curvature, the Codazzi equation
implies
\begin{equation}
\diver_gA
=
\grad_g(\tr_gA)
=
2\varepsilon\grad_gH.
\label{eq:divA}
\end{equation}
Moreover, by \eqref{eq:ambient-curvature-convention},
\begin{equation}
\sum_{i=1}^2\epsilon_i
R^N(e_i,H\xi)e_i
=-2cH\xi.
\label{eq:curvature-trace}
\end{equation}
Substitution of \eqref{eq:divA} and \eqref{eq:curvature-trace} into
\eqref{eq:L-operator} yields \eqref{eq:L-Hxi}.
\end{proof}

\begin{theorem}[Surface form of the biharmonic conformal equation]
\label{thm:unified-conformal-biharmonic}
Let
\begin{equation}
\phi:(M^2,\bar g)\longrightarrow N^3(c)
\end{equation}
be a spacelike or timelike conformal immersion, with
$g=\phi^*h=\lambda^2\bar g$.  Let $\xi$, $\varepsilon$, $A$, and $H$ be as in
\eqref{eq:epsilon-definition}--\eqref{eq:H-trA-relation}, and define
\begin{equation}
u:=\lambda^2H.
\label{eq:u-definition}
\end{equation}
Then $\phi:(M^2,\bar g)\to N^3(c)$ is biharmonic if and only if
\begin{equation}
\boxed{
\Delta_gu+
\left(-\varepsilon|A|_g^2+2c\right)u=0
}
\label{eq:unified-normal}
\end{equation}
and
\begin{equation}
\boxed{
A(\grad_gu)+\varepsilon u\grad_gH=0.
}
\label{eq:unified-tangential}
\end{equation}
\end{theorem}

\begin{proof}
By \eqref{eq:tension-isometric},
\begin{equation}
\tau_g(\phi)=2H\xi.
\end{equation}
Using Lemma~\ref{lem:normal-bitension}, equation
\eqref{eq:f-biharmonic-map-equation}, divided by $2$, becomes
\begin{equation}
f\mathcal L_g^\phi(H\xi)
+H(\Delta_gf)\xi
+2\nabla^\phi_{\grad_gf}(H\xi)=0.
\label{eq:pre-split}
\end{equation}
The last term satisfies
\begin{equation}
\nabla^\phi_{\grad_gf}(H\xi)
=
g(\grad_gf,\grad_gH)\xi
-HA(\grad_gf).
\label{eq:weighted-normal-derivative}
\end{equation}
The normal part of \eqref{eq:pre-split} is therefore
\begin{align}
0={}&
f\left(
\Delta_gH-\varepsilon|A|_g^2H+2cH
\right)
+H\Delta_gf
+2g(\grad_gf,\grad_gH)\\
={}&
\Delta_g(fH)
+
\left(-\varepsilon|A|_g^2+2c\right)fH.
\end{align}
Since $u=fH$, this is \eqref{eq:unified-normal}.

The tangential part is
\begin{align}
0
&=
-2fA(\grad_gH)
-2\varepsilon fH\grad_gH
-2HA(\grad_gf)\\
&=
-2\left[
A\bigl(\grad_g(fH)\bigr)
+\varepsilon fH\grad_gH
\right],
\end{align}
which is equivalent to \eqref{eq:unified-tangential}.
\end{proof}

\begin{corollary}[Spacelike surfaces in $\AdS_3(L)$]
\label{cor:spacelike-system}
Let $\phi:(M^2,\lambda^{-2}g)\to\AdS_3(L)$ be a spacelike conformal
immersion.  Then $g$ is Riemannian, $\varepsilon=-1$, and
$H=-\frac12\tr_gA$.  The immersion is biharmonic if and only if, for
$u=\lambda^2H$,
\begin{equation}
\boxed{
\Delta_gu+
\left(
|A|_g^2-\frac{2}{L^2}
\right)u=0
}
\label{eq:spacelike-normal}
\end{equation}
and
\begin{equation}
\boxed{
A(\grad_gu)-u\grad_gH=0.
}
\label{eq:spacelike-tangential}
\end{equation}
\end{corollary}

\begin{corollary}[Timelike surfaces in $\AdS_3(L)$]
\label{cor:timelike-system}
Let $\phi:(M^2,\lambda^{-2}g)\to\AdS_3(L)$ be a timelike conformal
immersion.  Then $g$ is Lorentzian, $\varepsilon=+1$, and
$H=\frac12\tr_gA$.  Writing $\Box_g=\Delta_g$ for the wave operator, the
immersion is biharmonic if and only if, for $u=\lambda^2H$,
\begin{equation}
\boxed{
\Box_gu-
\left(
|A|_g^2+\frac{2}{L^2}
\right)u=0
}
\label{eq:timelike-normal}
\end{equation}
and
\begin{equation}
\boxed{
A(\grad_gu)+u\grad_gH=0.
}
\label{eq:timelike-tangential}
\end{equation}
Here both $\Box_g$ and $|A|_g^2=\tr_g(A^2)$ are indefinite objects.
\end{corollary}

\begin{remark}[Orientation invariance]
Replacing $\xi$ by $-\xi$ sends $A$, $H$, and $u$ to $-A$, $-H$, and $-u$,
respectively.  Equations \eqref{eq:unified-normal} and
\eqref{eq:unified-tangential} are unchanged.  Thus the system is independent
of the chosen orientation.
\end{remark}

\begin{remark}[Alternative scalar mean-curvature convention]
Some authors define
\begin{equation}
H_{\mathrm{tr}}:=\frac12\tr_gA,
\end{equation}
so that $H=\varepsilon H_{\mathrm{tr}}$.  If
$v=\lambda^2H_{\mathrm{tr}}$, then the unified system becomes
\begin{equation}
\Delta_gv+
\left(-\varepsilon|A|_g^2+2c\right)v=0,
\qquad
A(\grad_gv)+v\grad_gH_{\mathrm{tr}}=0.
\label{eq:trace-H-convention}
\end{equation}
This explains why formulas written with different scalar mean-curvature
conventions can appear to have different tangential signs.
\end{remark}

\section{The nonminimal CMC branch is rigid}

Constant mean curvature is the natural first test of the conformal problem:
most classified proper biharmonic hypersurfaces lie on CMC or isoparametric
branches, and one might hope that a variable source metric creates further
examples.  The theorem below shows that this hope fails in the spacelike
$\AdS_3$ CMC branch.  The obstruction is local and therefore requires no
completeness or compactness assumption.

\begin{theorem}[Nonzero CMC forces constant dilation]
\label{thm:CMC-rigidity}
Let $M^2$ be connected and let
\begin{equation}
\phi:(M^2,\bar g)\longrightarrow\AdS_3(L)
\end{equation}
be a two-sided spacelike conformal immersion satisfying
\begin{equation}
g=\phi^*h=\lambda^2\bar g,
\qquad
\lambda>0.
\end{equation}
Assume that $\phi:(M^2,\bar g)\to\AdS_3(L)$ is biharmonic and that its scalar
mean curvature, defined by $\mathbf H=H\xi$, is a nonzero constant,
\begin{equation}
H\equiv H_0\neq0.
\end{equation}
Then $\lambda$ is constant on $M$.  Moreover,
\begin{equation}
|A|_g^2=\frac{2}{L^2}.
\label{eq:CMC-A-square}
\end{equation}
\end{theorem}

\begin{proof}
Set
\begin{equation}
f=\lambda^2,
\qquad
u=fH_0.
\end{equation}
Since $H$ is constant, the tangential equation
\eqref{eq:spacelike-tangential} reduces to
\begin{equation}
A(\grad_gu)=0.
\label{eq:A-grad-u-zero}
\end{equation}
Suppose, toward a contradiction, that $u$ is nonconstant.  Since $M$ is
connected, there exists a point at which $\grad_gu\neq0$.  On a sufficiently
small open neighbourhood $U$ of such a point, choose a local orthonormal frame
$\{e_1,e_2\}$ with
\begin{equation}
e_1=\frac{\grad_gu}{|\grad_gu|_g}.
\label{eq:e1-gradient}
\end{equation}
Equation \eqref{eq:A-grad-u-zero} gives
\begin{equation}
Ae_1=0.
\label{eq:Ae1-zero}
\end{equation}
Because $A$ is self-adjoint with respect to the Riemannian metric $g$, the
orthogonal direction is also principal; hence
\begin{equation}
Ae_2=k e_2
\label{eq:Ae2-k}
\end{equation}
for a smooth function $k$ on $U$.  In the spacelike convention,
$H=-\frac12\tr_gA$, and therefore
\begin{equation}
k=\tr_gA=-2H_0.
\label{eq:k-constant}
\end{equation}
Thus $k$ is a nonzero constant.

Let $\omega$ be the connection one-form of the orthonormal frame, defined by
\begin{equation}
\nabla_Xe_1=\omega(X)e_2,
\qquad
\nabla_Xe_2=-\omega(X)e_1.
\label{eq:connection-form}
\end{equation}
The Codazzi equation in a space form is
\begin{equation}
(\nabla_{e_1}A)e_2=(\nabla_{e_2}A)e_1.
\label{eq:codazzi-principal}
\end{equation}
Using \eqref{eq:Ae1-zero}, \eqref{eq:Ae2-k}, and the constancy of $k$, we
compute
\begin{align}
(\nabla_{e_1}A)e_2
&=
-k\omega(e_1)e_1,
\label{eq:codazzi-left}\\
(\nabla_{e_2}A)e_1
&=
-k\omega(e_2)e_2.
\label{eq:codazzi-right}
\end{align}
Since $k\neq0$ and $e_1,e_2$ are linearly independent,
\eqref{eq:codazzi-principal} implies
\begin{equation}
\omega(e_1)=\omega(e_2)=0.
\label{eq:omega-zero}
\end{equation}
Hence $\omega=0$ on $U$, the frame is parallel, and the Gaussian curvature of
$(U,g)$ is
\begin{equation}
K=0.
\label{eq:K-zero}
\end{equation}

On the other hand, the Gauss equation for a spacelike surface with timelike
unit normal in a Lorentzian space form is
\begin{equation}
K=c+\varepsilon\det A
=c-\det A.
\label{eq:gauss-spacelike}
\end{equation}
Here $c=-L^{-2}$ and, by \eqref{eq:Ae1-zero}, $\det A=0$.  Therefore
\begin{equation}
K=-\frac{1}{L^2},
\label{eq:K-negative}
\end{equation}
contradicting \eqref{eq:K-zero}.  It follows that $\grad_gu$ vanishes
everywhere.  Since $M$ is connected, $u$ is constant.  As $H_0\neq0$,
\begin{equation}
f=\frac{u}{H_0}=\lambda^2
\end{equation}
is constant, and therefore $\lambda$ is constant.

Finally, $u$ is a nonzero constant.  Substitution into the normal equation
\eqref{eq:spacelike-normal} gives
\begin{equation}
\left(|A|_g^2-\frac{2}{L^2}\right)u=0.
\end{equation}
Since $u\neq0$, equation \eqref{eq:CMC-A-square} follows.
\end{proof}

\begin{corollary}[Geometry of the nonminimal CMC branch]
\label{cor:CMC-branch-classification}
Under the hypotheses of Theorem~\ref{thm:CMC-rigidity}, the associated induced
isometric immersion
\begin{equation}
\phi:(M^2,g)\longrightarrow\AdS_3(L)
\end{equation}
is biharmonic and locally totally umbilical.  More precisely,
\begin{equation}
A=\pm\frac1L\operatorname{id},
\qquad
H=\mp\frac1L,
\qquad
K=-\frac{2}{L^2}.
\label{eq:CMC-umbilical-classification}
\end{equation}
Consequently, up to reversal of the unit normal and an ambient isometry, every
point has a neighbourhood congruent to an open subset of the standard
hyperbolic plane $\mathbb H^2(L/\sqrt2)\subset\AdS_3(L)$.  Thus the nonminimal
CMC conformal-biharmonic branch consists only of constant homothetic
reparametrizations of the standard proper biharmonic isometric branch.
\end{corollary}

\begin{proof}
When $f=\lambda^2$ is constant, equation
\eqref{eq:f-biharmonic-map-equation} reduces to
$f\tau_{2,g}(\phi)=0$.  Since $f>0$, one has $\tau_{2,g}(\phi)=0$.

It remains to identify the geometry.  Let $k_1,k_2$ denote the principal
curvatures on a local principal neighbourhood.  Theorem~\ref{thm:CMC-rigidity}
gives
\begin{equation}
k_1+k_2=-2H_0,
\qquad
k_1^2+k_2^2=\frac{2}{L^2}.
\end{equation}
Hence both the sum and product of $k_1,k_2$ are constant.  Suppose that the
nonumbilic set $\{k_1\neq k_2\}$ were nonempty, and choose a connected
principal neighbourhood contained in one of its components.  On this
neighbourhood the two principal curvatures are the distinct roots of a fixed
quadratic polynomial, and hence are individually constant.  The Codazzi
equations would then force the principal frame to be parallel, so $K=0$.
The Gauss equation would give
$k_1k_2=-L^{-2}$, and therefore
\begin{equation}
(k_1+k_2)^2=k_1^2+k_2^2+2k_1k_2=0,
\end{equation}
contradicting $H_0\neq0$.  Thus $k_1=k_2=k$ everywhere.  It follows that
$2k^2=2L^{-2}$, so $k=\pm L^{-1}$; connectedness fixes the sign.  Since
$H=-k$ in the spacelike convention, the first two identities in
\eqref{eq:CMC-umbilical-classification} follow, and the Gauss equation yields
$K=-L^{-2}-k^2=-2L^{-2}$.  The local congruence statement follows from the
fundamental theorem of hypersurfaces in a space form.
\end{proof}

\medskip
\noindent\textbf{Interpretation.}
The CMC conclusion is stronger than a restriction on one ansatz.  It says
that no local two-sided spacelike CMC surface can use a nonconstant conformal
factor to enter the proper biharmonic class.  The only surviving nonminimal
geometry is the totally umbilical branch already visible in the isometric
problem.  Consequently, every genuinely new conformal example must leave
constant mean curvature; this is exactly the transition made in the next
section.

\begin{remark}[A genuinely Lorentzian nonpositive-curvature branch]
\label{rem:lorentzian-riemannian-contrast}
The conclusion of Corollary~\ref{cor:CMC-branch-classification} sharply
contrasts with the corresponding Riemannian nonpositive-curvature picture.
For Riemannian space forms, Mohammed Cherif and Ou prove that no part of a
nonminimal totally umbilical hypersurface in a space form of nonpositive
curvature admits a biharmonic conformal immersion
\cite{CherifOu2026}; related $f$-biharmonic nonexistence statements for
totally umbilical surfaces in nonpositively curved Riemannian
three-manifolds appear in \cite{WangQinChen2024}.  In the present spacelike
$\AdS_3$ problem the unit normal is timelike, so $\varepsilon=-1$ and the
shape-operator contribution in
\eqref{eq:unified-normal} changes sign.  The surviving branch
\[
A=\pm L^{-1}\operatorname{id},
\qquad
K=-2L^{-2},
\]
is therefore not a Riemannian negative-curvature analogue in disguise; it is
a causal-sign effect intrinsic to the Lorentzian ambient geometry.
\end{remark}

\section{Leaving the rigid branch: open local non-CMC families}
\label{sec:cohomogeneity-one-existence}

The CMC theorem leaves two possibilities: either nonconstant dilation is
impossible altogether, or the constant-mean-curvature assumption is the
source of the rigidity.  This section proves the second alternative.  We use
one symmetry variable to reduce the surface equations to an analytic initial
value problem and show that nonconstant $H$ and nonconstant $\lambda$ occur
on an open set of data.

The construction is local but fully geometric.  A solution of the reduced
ODE determines first and second fundamental forms satisfying the
Gauss--Codazzi equations; the fundamental theorem of hypersurfaces then
produces a local immersion into anti-de Sitter space, unique up to ambient
isometry \cite[Chapter~7]{ONeill1983}.  The point is not merely to display
one exceptional solution, but to prove that the non-CMC regime is stable
under perturbation of its initial profile.

\subsection{A symmetry reduction adapted to principal curvature lines}

At this stage, $t$ is only a symmetry coordinate for the intrinsic metric and
second fundamental form.  Section~\ref{sec:analytical-rotational-reduction}
identifies the additional ambient moving-frame condition and proves when the
$t$-translations are generated by a one-parameter subgroup of
$\operatorname{Isom}(\AdS_3(L))$.

Let $I\subset\mathbb R$ be an interval and consider
\begin{equation}
M=I\times J,
\qquad
g=ds^2+\rho(s)^2dt^2,
\qquad
\rho>0.
\label{eq:spacelike-warped-metric}
\end{equation}
Introduce the oriented orthonormal frame
\begin{equation}
e_1=\partial_s,
\qquad
e_2=\rho^{-1}\partial_t,
\qquad
q:=\frac{\rho'}{\rho}.
\label{eq:spacelike-principal-frame}
\end{equation}
Assume that the frame is principal for the shape operator,
\begin{equation}
Ae_1=k_1(s)e_1,
\qquad
Ae_2=k_2(s)e_2,
\label{eq:principal-shape-ansatz}
\end{equation}
and that the weighted mean-curvature function depends only on $s$,
\begin{equation}
u=u(s)=\lambda^2H.
\end{equation}
The Levi--Civita connection of \eqref{eq:spacelike-warped-metric} is
\begin{equation}
\nabla_{e_1}e_1=0,
\qquad
\nabla_{e_1}e_2=0,
\qquad
\nabla_{e_2}e_1=qe_2,
\qquad
\nabla_{e_2}e_2=-qe_1,
\label{eq:spacelike-warped-connection}
\end{equation}
and therefore
\begin{equation}
K=-\frac{\rho''}{\rho}=-(q'+q^2).
\label{eq:warped-gaussian-curvature}
\end{equation}
Since the unit normal is timelike, the Gauss equation is
\begin{equation}
K=-\frac1{L^2}-k_1k_2.
\label{eq:spacelike-gauss-principal}
\end{equation}
Equations \eqref{eq:warped-gaussian-curvature} and
\eqref{eq:spacelike-gauss-principal} give
\begin{equation}
q'=\frac1{L^2}+k_1k_2-q^2.
\label{eq:q-spacelike}
\end{equation}
The Codazzi equation gives
\begin{equation}
k_2'=q(k_1-k_2).
\label{eq:k2-codazzi}
\end{equation}
Our spacelike mean-curvature convention is
\begin{equation}
H=-\frac{k_1+k_2}{2}.
\label{eq:H-principal-spacelike}
\end{equation}
Moreover,
\begin{equation}
\grad_gu=u'e_1,
\qquad
\Delta_gu=u''+qu'.
\label{eq:spacelike-radial-laplacian}
\end{equation}
Thus the normal conformal-biharmonic equation
\eqref{eq:spacelike-normal} becomes
\begin{equation}
u''+qu'
+
\left(k_1^2+k_2^2-\frac{2}{L^2}\right)u=0,
\label{eq:cohom1-normal-equation}
\end{equation}
whereas the tangential equation \eqref{eq:spacelike-tangential} is
\begin{equation}
k_1u'=uH'.
\label{eq:cohom1-tangential-equation}
\end{equation}
Using \eqref{eq:H-principal-spacelike}, this is equivalent to
\begin{equation}
k_1'+k_2'=-2k_1\frac{u'}{u}
\label{eq:sum-principal-equation}
\end{equation}
wherever $u\neq0$.

Set
\begin{equation}
p:=u'.
\end{equation}
Combining
\eqref{eq:q-spacelike}, \eqref{eq:k2-codazzi},
\eqref{eq:cohom1-normal-equation}, and
\eqref{eq:sum-principal-equation} gives the first-order system
\begin{equation}
\boxed{
\begin{aligned}
\rho'&=q\rho,\\
u'&=p,\\
p'&=-qp-
\left(k_1^2+k_2^2-\frac{2}{L^2}\right)u,\\
k_2'&=q(k_1-k_2),\\
k_1'&=-2k_1\frac{p}{u}-q(k_1-k_2),\\
q'&=\frac1{L^2}+k_1k_2-q^2.
\end{aligned}}
\label{eq:spacelike-cohom1-system}
\end{equation}
This is the basic first-order cohomogeneity-one system in the present sign
conventions.

It is often numerically preferable to introduce
\begin{equation}
w:=\frac{p}{u}=(\log|u|)'.
\label{eq:w-definition}
\end{equation}
Then the geometric variables satisfy the autonomous system
\begin{equation}
\boxed{
\begin{aligned}
k_1'&=-2k_1w-q(k_1-k_2),\\
k_2'&=q(k_1-k_2),\\
q'&=\frac1{L^2}+k_1k_2-q^2,\\
w'&=-qw-w^2-
\left(k_1^2+k_2^2-\frac{2}{L^2}\right),
\end{aligned}}
\label{eq:reduced-autonomous-system}
\end{equation}
followed by the quadratures
\begin{equation}
(\log\rho)'=q,
\qquad
(\log|u|)'=w.
\label{eq:rho-u-quadratures}
\end{equation}

\subsection{Analytic initial data and robust nonconstant solutions}

\begin{theorem}[Local solutions from analytic initial data]
\label{thm:cohom1-local-existence}
Fix $L>0$ and initial data
\begin{equation}
\rho(0)=\rho_0>0,
\quad
u(0)=u_0\neq0,
\quad
p(0)=p_0,
\quad
k_1(0)=a_0,
\quad
k_2(0)=b_0,
\quad
q(0)=q_0.
\label{eq:general-initial-data}
\end{equation}
Then there is a unique real-analytic solution of
\eqref{eq:spacelike-cohom1-system} on some interval
$(-\delta,\delta)$.  Define
\begin{equation}
H=-\frac{k_1+k_2}{2}.
\end{equation}
If
\begin{equation}
H_0:=-\frac{a_0+b_0}{2}\neq0,
\qquad
\frac{u_0}{H_0}>0,
\label{eq:admissible-initial-data}
\end{equation}
then, after decreasing $\delta$ if necessary,
\begin{equation}
\lambda^2:=\frac{u}{H}>0
\end{equation}
and there exists a spacelike immersion
\begin{equation}
\phi:
\bigl(( -\delta,\delta)\times J,g\bigr)
\longrightarrow \AdS_3(L)
\end{equation}
with induced metric \eqref{eq:spacelike-warped-metric}, shape operator
\eqref{eq:principal-shape-ansatz}, and timelike unit normal.  The conformally
reparametrized immersion
\begin{equation}
\phi:
\bigl(( -\delta,\delta)\times J,\lambda^{-2}g\bigr)
\longrightarrow \AdS_3(L)
\end{equation}
is biharmonic.  It is proper biharmonic, in the sense of being
nonharmonic, because
\begin{equation}
\tau_{\lambda^{-2}g}(\phi)=2u\xi\neq0.
\end{equation}
Locally, the immersion is unique up to an ambient isometry of
$\AdS_3(L)$.
\end{theorem}

\begin{proof}
The right-hand side of \eqref{eq:spacelike-cohom1-system} is real analytic on
\begin{equation}
\Omega=
\{(\rho,u,p,k_1,k_2,q):\rho>0,\ u\neq0\}.
\end{equation}
The analytic ordinary differential equation theorem therefore gives a unique
real-analytic solution through every initial point in $\Omega$.

Equations \eqref{eq:q-spacelike} and \eqref{eq:k2-codazzi} are precisely the
Gauss and Codazzi equations for the metric
\eqref{eq:spacelike-warped-metric} and the self-adjoint field
\eqref{eq:principal-shape-ansatz}.  On a simply connected coordinate
neighbourhood, the fundamental theorem of hypersurfaces in a
semi-Riemannian space form therefore produces the required spacelike
immersion into $\AdS_3(L)$, unique up to ambient isometry.

Because $H_0\neq0$ and $u_0/H_0>0$, continuity allows us to shrink the
interval so that $H$ has no zero and $u/H$ remains positive.  Finally, $u$ remains nonzero after another possible restriction of the
interval.  Therefore the conformally changed tension field is nonzero.
Equations
\eqref{eq:cohom1-normal-equation} and
\eqref{eq:cohom1-tangential-equation} are exactly the normal and tangential
conformal-biharmonic equations.  Hence the immersion with domain metric
$\lambda^{-2}g$ is biharmonic.
\end{proof}

The preceding theorem becomes a genuine non-CMC existence theorem on an open
set of initial data.

\begin{corollary}[An open non-CMC, nonconstant-dilation family]
\label{cor:open-nonconstant-family}
In addition to \eqref{eq:admissible-initial-data}, assume
\begin{equation}
a_0p_0\neq0,
\qquad
p_0(H_0-a_0)\neq0.
\label{eq:open-nonconstant-conditions}
\end{equation}
Then both the mean curvature $H$ and the conformal dilation $\lambda$ are
nonconstant on every sufficiently small neighbourhood of $s=0$.
The conditions
\eqref{eq:admissible-initial-data}--\eqref{eq:open-nonconstant-conditions}
define an open subset of the initial-data space.
\end{corollary}

\begin{proof}
Equation \eqref{eq:cohom1-tangential-equation} gives
\begin{equation}
H'(0)=a_0\frac{p_0}{u_0}.
\label{eq:Hprime-initial}
\end{equation}
Furthermore,
\begin{align}
(\lambda^2)'
&=
\left(\frac{u}{H}\right)'
=
\frac{pH-uH'}{H^2}
=
\frac{p(H-k_1)}{H^2},
\label{eq:lambda-prime-general}
\end{align}
and hence
\begin{equation}
(\lambda^2)'(0)
=
\frac{p_0(H_0-a_0)}{H_0^2}.
\label{eq:lambda-prime-initial}
\end{equation}
The two inequalities in \eqref{eq:open-nonconstant-conditions} make these
derivatives nonzero.  All imposed conditions are strict inequalities and are
therefore open.
\end{proof}

\medskip
\noindent\textbf{Why openness matters.}
The inequalities in Corollary~\ref{cor:open-nonconstant-family} are strict,
so the new behavior survives small changes of the initial data.  The
nonconstant-dilation solutions are therefore not a single finely tuned
counterexample to CMC rigidity; they occupy a genuine open region of the
local solution space.

\begin{proposition}[Local parameter count]
\label{prop:cohom1-parameter-count}
Fix $L>0$ and consider pointed local germs in the open non-CMC branch of
Corollary~\ref{cor:open-nonconstant-family}.  Modulo ambient isometry,
positive rescaling of the orbit coordinate, and constant homothety of the
source metric, the generic cohomogeneity-one germ is locally determined by
the four reduced initial quantities
\begin{equation}
(k_{1,0},k_{2,0},q_0,w_0),
\qquad
w_0=\frac{p_0}{u_0}.
\label{eq:four-reduced-moduli}
\end{equation}
Equivalently, the open family contains a four-parameter family of reduced
profile germs after these natural gauges are removed.
\end{proposition}

\begin{proof}
The autonomous system \eqref{eq:reduced-autonomous-system} shows that the
reduced geometric profile
\[
(k_1,k_2,q,w)
\]
is uniquely determined by the four values in
\eqref{eq:four-reduced-moduli}.  The remaining quadratures
\eqref{eq:rho-u-quadratures} introduce the positive constants
$\rho_0$ and $|u_0|$.  A positive change of orbit coordinate
$\widetilde t=\alpha t$ replaces $\rho$ by $\rho/\alpha$ and may be
used to normalize $\rho_0$.  Multiplying $u$ by a positive constant
multiplies $\lambda^2=u/H$ by the same constant and therefore replaces the
source metric $\bar g=\lambda^{-2}g$ by a constant homothety; it leaves
$(k_1,k_2,q,w)$ unchanged.  Finally, the fundamental theorem of
hypersurfaces already identifies immersions with the same first and second
fundamental forms up to ambient isometry.  Thus, for pointed germs, the four
reduced initial quantities are the continuous profile parameters left after
the stated gauges are removed.
\end{proof}

\begin{remark}[Meaning of the dimension count]
The qualifier ``pointed'' fixes the reference value $s=0$.  Passing to
unpointed germs introduces the usual translation freedom in the autonomous
profile parameter; no such quotient is needed for the initial-value
statements used in this paper.  Residual discrete operations, including
normal reversal, profile reversal, and orbit reversal, may identify reduced
initial data, and special data can have larger stabilizers.  Thus the number
four is the continuous dimension of the generic free parameter stratum (or
local orbifold stratum), not a claim that the full quotient is everywhere a
smooth moduli manifold.
\end{remark}

\subsection{A concrete local solution germ}

For $L=1$, choose
\begin{equation}
\rho_0=1,
\qquad
u_0=1,
\qquad
p_0=1,
\qquad
k_{1,0}=-1,
\qquad
k_{2,0}=-2,
\qquad
q_0=0.
\label{eq:strong-local-initial-data}
\end{equation}
Then
\begin{equation}
H_0=\frac32,
\qquad
\lambda_0^2=\frac23,
\qquad
H'(0)=-1,
\qquad
(\lambda^2)'(0)=\frac{10}{9}.
\end{equation}
The solution supplied by Theorem~\ref{thm:cohom1-local-existence} has the
Taylor expansions
\begin{align}
u(s)
&=1+s-\frac32s^2+O(s^3),
\label{eq:u-local-series}\\
k_1(s)
&=-1+2s-\frac{15}{2}s^2+O(s^3),
\label{eq:k1-local-series}\\
k_2(s)
&=-2+\frac32s^2+O(s^3),
\label{eq:k2-local-series}\\
q(s)
&=3s-2s^2+O(s^3),
\label{eq:q-local-series}\\
\rho(s)
&=1+\frac32s^2-\frac23s^3+O(s^4),
\label{eq:rho-local-series}\\
H(s)
&=\frac32-s+3s^2+O(s^3),
\label{eq:H-local-series}\\
\lambda^2(s)
&=\frac23+\frac{10}{9}s-\frac{43}{27}s^2+O(s^3).
\label{eq:lambda-local-series}
\end{align}
Thus this initial condition produces an actual analytic proper biharmonic
conformal immersion for which both $H$ and $\lambda$ are nonconstant.

\section{From reduced data to genuine rotational surfaces}
\label{sec:analytical-rotational-reduction}

Section~
ef{sec:cohomogeneity-one-existence} constructs surfaces from
intrinsic metric and curvature data, but it does not yet tell the reader
which ambient isometries generate the symmetry or how to write the surface
itself in coordinates.  This section supplies that missing extrinsic step.

The moving frame produces a constant element of $\mathfrak{so}(2,2)$ whose
flow generates the orbit direction.  A single conserved scalar determines
whether that flow is elliptic, hyperbolic, or parabolic.  The parabolic case
is then especially tractable: the geometric variables collapse to one
third-order scalar equation, and the immersion can be reconstructed by
quadratures.  This is the point at which the existence theory becomes a
concrete family of ambient surfaces.

\subsection{The moving frame and the ambient symmetry generator}

Let
\begin{equation}
    \operatorname{AdS}_3(L)
    =
    \left\{
        X\in\mathbb{R}^{2,2}:
        \langle X,X\rangle_{2,2}=-L^2
    \right\},
\end{equation}
where
\begin{equation}
    \langle X,X\rangle_{2,2}
    =
    -X_0^2-X_1^2+X_2^2+X_3^2.
\end{equation}
Consider a nondegenerate cohomogeneity-one immersion
\begin{equation}
    X:I\times J\longrightarrow\operatorname{AdS}_3(L)
\end{equation}
whose induced metric is
\begin{equation}
    g
    =
    \sigma\,ds^2+\eta\,\rho(s)^2dt^2,
    \qquad
    \sigma,\eta\in\{-1,+1\},
    \qquad
    \rho>0.
    \label{eq:rot-metric-analytical}
\end{equation}
Define the pseudo-orthonormal tangent frame
\begin{equation}
    e_1=\partial_s,
    \qquad
    e_2=\rho^{-1}\partial_t,
\end{equation}
so that
\begin{equation}
    \langle e_1,e_1\rangle=\sigma,
    \qquad
    \langle e_2,e_2\rangle=\eta.
\end{equation}
Let $\xi$ be a unit normal and put
\begin{equation}
    \varepsilon
    =
    \langle\xi,\xi\rangle
    =
    -\sigma\eta.
    \label{eq:epsilon-sign-relation}
\end{equation}
We assume that the orbit and profile directions are principal:
\begin{equation}
    Ae_1=k_1e_1,
    \qquad
    Ae_2=k_2e_2.
    \label{eq:rot-principal-directions}
\end{equation}
Set
\begin{equation}
    q=\frac{\rho'}{\rho}.
\end{equation}
The scalar mean curvature, with
$\mathbf H=H\xi$, is
\begin{equation}
    H
    =
    \frac{\varepsilon}{2}(k_1+k_2).
    \label{eq:rot-H-analytical}
\end{equation}

The Levi--Civita connection of \eqref{eq:rot-metric-analytical} is
\begin{align}
    \nabla_{e_1}e_1&=0,
    &
    \nabla_{e_1}e_2&=0,
    \\
    \nabla_{e_2}e_1&=qe_2,
    &
    \nabla_{e_2}e_2&=-\sigma\eta q e_1.
\end{align}
Consequently, the Gauss and Codazzi equations are
\begin{align}
    q'
    &=
    \sigma
    \left(
        \frac{1}{L^2}-\varepsilon k_1k_2
    \right)
    -q^2,
    \label{eq:rot-gauss-analytical}\\
    k_2'
    &=
    q(k_1-k_2).
    \label{eq:rot-codazzi-analytical}
\end{align}

Let
\begin{equation}
    E_0=\frac{X}{L}.
\end{equation}
Then
\begin{equation}
    \mathcal F=(E_0,e_1,e_2,\xi)
\end{equation}
is a pseudo-orthonormal ambient frame with Gram matrix
\begin{equation}
    \operatorname{diag}(-1,\sigma,\eta,\varepsilon).
\end{equation}
Using the flat ambient connection $D$ and the Gauss--Weingarten
relations, differentiation in the orbit direction gives
\begin{align}
    D_tE_0
    &=
    \frac{\rho}{L}e_2,
    \\
    D_te_1
    &=
    \rho q e_2,
    \\
    D_te_2
    &=
    \rho
    \left(
        \frac{\eta}{L}E_0
        -\sigma\eta q e_1
        +\varepsilon\eta k_2\xi
    \right),
    \\
    D_t\xi
    &=
    -\rho k_2e_2.
\end{align}
Equivalently,
\begin{equation}
    \mathcal F_t=\mathcal F M,
    \label{eq:F-t-M}
\end{equation}
where
\begin{equation}
    M
    =
    \rho
    \begin{pmatrix}
        0 & 0 & \eta/L & 0\\
        0 & 0 & -\sigma\eta q & 0\\
        1/L & q & 0 & -k_2\\
        0 & 0 & \varepsilon\eta k_2 & 0
    \end{pmatrix}.
    \label{eq:orbit-generator-matrix}
\end{equation}

\begin{lemma}[Rank-two normal forms needed below]
\label{lem:rank-two-normal-forms}
Let $(V,\langle\cdot,\cdot\rangle)$ have signature $(2,2)$ and let
$B\in\mathfrak{so}(V)$ have rank two.

\begin{enumerate}
\item If $B^3=\chi B$ with $\chi\neq0$, then
$W:=\operatorname{im}B$ is nondegenerate,
$\ker B=W^\perp$, and $B^2=\chi\operatorname{id}$ on $W$.  If
$\chi<0$, the plane $W$ is definite and $B|_W$ is a rotation after
scaling; if $\chi>0$, the plane $W$ is Lorentzian and $B|_W$ is a boost
after scaling.

\item If $B^3=0$ and $B^2\neq0$, then $B$ preserves a null flag
\[
0\subset\operatorname{im}B^2\subset\operatorname{im}B
\subset\ker B^2\subset V.
\]
After an ambient orthogonal change of basis and a nonzero rescaling of the
orbit parameter, $B$ is one of the two causal index-three null-rotation
blocks.  In null coordinates with metric
$-2\,dU\,dV-dY^2+dZ^2$, representatives are
\begin{equation}
B_{\mathrm{par}}^{(+)}
=
\begin{pmatrix}
0&0&0&0\\
0&0&0&1\\
0&0&0&0\\
1&0&0&0
\end{pmatrix},
\qquad
B_{\mathrm{par}}^{(-)}
=
\begin{pmatrix}
0&0&0&0\\
0&0&-1&0\\
1&0&0&0\\
0&0&0&0
\end{pmatrix}.
\label{eq:parabolic-normal-form}
\end{equation}
The sign records whether the non-null direction in $\operatorname{im}B$
is spacelike or timelike.

\item If $B^2=0$, then
$\operatorname{im}B=\ker B$ is a totally null two-plane.  This is a
distinct rank-two nilpotent type and is not conjugate to either block in
item~2.  For example,
\begin{equation}
B_0=
\begin{pmatrix}
0&0&-1&1\\
0&0&0&0\\
0&1&0&0\\
0&1&0&0
\end{pmatrix}
\label{eq:square-zero-nilpotent}
\end{equation}
is skew-adjoint for $-2\,dU\,dV-dY^2+dZ^2$, has rank two, and satisfies
$B_0^2=0$.
\end{enumerate}
\end{lemma}

\begin{proof}
Skew-adjointness gives
$\ker B=(\operatorname{im}B)^\perp$.  Suppose first that $\chi\neq0$.
If $v\in W\cap W^\perp$, write $v=Bx$.  Then $Bv=0$, so
$B^2x=0$ and
$0=B^3x=\chi Bx=\chi v$; hence $v=0$.  Thus $W$ is nondegenerate and
$V=W\oplus W^\perp$.  Since $B^3=\chi B$ and $B|_W$ is onto, one has
$B^2=\chi\operatorname{id}$ on $W$.

Choose a pseudo-orthonormal basis $e_1,e_2$ of $W$, with
$\langle e_i,e_i\rangle=\delta_i\in\{\pm1\}$.  Skew-adjointness forces
\[
Be_1=\alpha e_2,
\qquad
Be_2=-\alpha\delta_1\delta_2 e_1,
\]
for some $\alpha\neq0$.  Consequently
$\chi=-\alpha^2\delta_1\delta_2$.  Hence $\chi<0$ precisely when $W$ is
definite, giving the elliptic rotation block, and $\chi>0$ precisely when
$W$ has signature $(1,1)$, giving the hyperbolic boost block.  The
endomorphism vanishes on $W^\perp$.

Now suppose $B^3=0$ and $B^2\neq0$.  Choose $x$ with
$n:=B^2x\neq0$ and put $z:=Bx$.  Then
$Bn=0$, $Bz=n$, and skew-adjointness gives
\[
\langle n,n\rangle=0,
\qquad
\langle z,n\rangle=0,
\qquad
\langle x,z\rangle=0.
\]
The plane $\operatorname{im}B=\operatorname{span}\{z,n\}$ cannot be
totally null: otherwise
$\operatorname{im}B=(\operatorname{im}B)^\perp=\ker B$, contradicting
$Bz=n\neq0$.  Thus $z$ is non-null.  After rescaling $x$, set
$\delta:=\langle z,z\rangle\in\{+1,-1\}$.  Skew-adjointness also gives
$\langle x,n\rangle=-\delta$.  Replacing $x$ by
\[
\widetilde x=x+\frac{\langle x,x\rangle}{2\delta}\,n
\]
preserves $B\widetilde x=z$ and makes $\widetilde x$ null.  Choose
$m\in\ker B$ independent of $n$, adjust it by a multiple of $n$ so that
$m\perp\widetilde x$, and normalize it.  The signature $(2,2)$ forces
$\langle m,m\rangle=-\delta$.  If $\delta=+1$, the ordered basis
$(\widetilde x,n,m,z)$ gives $B_{\mathrm{par}}^{(+)}$; if
$\delta=-1$, the ordered basis $(\widetilde x,-n,z,m)$ gives
$B_{\mathrm{par}}^{(-)}$.  The displayed null flag follows immediately.

Finally, if $B^2=0$, then
$\operatorname{im}B\subset\ker B$.  Both spaces have dimension two, so
they coincide.  Since
$\ker B=(\operatorname{im}B)^\perp$, this common plane equals its
orthogonal complement and is therefore totally null.  The matrix
\eqref{eq:square-zero-nilpotent} verifies that this case occurs; its
nilpotency index distinguishes it from item~2.
\end{proof}

\begin{theorem}[The conserved ambient generator and its orbit type]
\label{thm:rotational-first-integral}
Let $X:I\times J\to\AdS_3(L)$ be a local nondegenerate
cohomogeneity-one immersion with metric, principal frame, ambient frame
$\mathcal F$, and matrices $M,N$ as in
\eqref{eq:rot-metric-analytical}--\eqref{eq:profile-generator-matrix}.
Assume that its Gauss--Codazzi equations
\eqref{eq:rot-gauss-analytical}--\eqref{eq:rot-codazzi-analytical} hold.
Then
\begin{equation}
    \boxed{
    \mathfrak c
    =
    \eta\rho^2
    \left(
        \frac{1}{L^2}
        -\sigma q^2
        -\varepsilon k_2^2
    \right)
    }
    \label{eq:orbit-invariant}
\end{equation}
is constant.

Moreover, there exists a constant element
$\mathcal B\in\mathfrak{so}(2,2)$ such that
\begin{equation}
    X_t=\mathcal B X,
    \qquad
    X(s,t)=e^{t\mathcal B}X(s,0).
    \label{eq:rotational-reconstruction}
\end{equation}
The generator satisfies
\begin{equation}
    \mathcal B^3=\mathfrak c\,\mathcal B.
    \label{eq:minimal-polynomial-B}
\end{equation}
The generator has rank two, and its rank-two causal normal-form type is
\begin{equation}
    \begin{cases}
        \mathfrak c<0
        &\Longleftrightarrow
        \text{elliptic rotation on a definite two-plane},\\[2mm]
        \mathfrak c>0
        &\Longleftrightarrow
        \text{hyperbolic boost on a Lorentzian two-plane},\\[2mm]
        \mathfrak c=0
        &\Longleftrightarrow
        \text{index-three parabolic null rotation}.
    \end{cases}
    \label{eq:orbit-type-sign}
\end{equation}
\end{theorem}

\begin{proof}
Put
\[
S=L^{-2}-\sigma q^2-\varepsilon k_2^2.
\]
Since $\mathfrak c=\eta\rho^2S$ and $\rho'=q\rho$,
\begin{align}
\frac{\mathfrak c'}{2\eta\rho^2}
&=
qS-\sigma q q'-\varepsilon k_2k_2'\\
&=
q\left(L^{-2}-\sigma q^2-\varepsilon k_2^2\right)
-\sigma q\left[
\sigma\left(L^{-2}-\varepsilon k_1k_2\right)-q^2
\right]
-\varepsilon qk_2(k_1-k_2)\\
&=0.
\end{align}

For a more geometric proof, the profile-direction frame equation is
\begin{equation}
    \mathcal F_s=\mathcal F N,
    \qquad
    N=
    \begin{pmatrix}
        0 & \sigma/L & 0 & 0\\
        1/L & 0 & 0 & -k_1\\
        0 & 0 & 0 & 0\\
        0 & \varepsilon\sigma k_1 & 0 & 0
    \end{pmatrix}.
    \label{eq:profile-generator-matrix}
\end{equation}
The compatibility condition
$\mathcal F_{st}=\mathcal F_{ts}$ is
\begin{equation}
    M'+NM-MN=0.
    \label{eq:zero-curvature-MN}
\end{equation}
Define the ambient endomorphism
\begin{equation}
    \mathcal B
    =
    \mathcal F M\mathcal F^{-1}.
\end{equation}
Equation \eqref{eq:zero-curvature-MN} implies
\begin{equation}
    \partial_s\mathcal B=0,
    \qquad
    \partial_t\mathcal B=0.
\end{equation}
Thus $\mathcal B$ is constant.  Since the first column of
\eqref{eq:orbit-generator-matrix} gives
\begin{equation}
    \mathcal B E_0=\frac{\rho}{L}e_2,
\end{equation}
we obtain $X_t=\mathcal B X$ and hence
\eqref{eq:rotational-reconstruction}.

The matrix $M$ is skew-adjoint with respect to
$\operatorname{diag}(-1,\sigma,\eta,\varepsilon)$, and therefore its
ambient conjugate $\mathcal B$ belongs to $\mathfrak{so}(2,2)$.  A direct
multiplication of \eqref{eq:orbit-generator-matrix} yields
\begin{equation}
    M^3=\mathfrak c\,M.
\end{equation}
Conjugating by $\mathcal F$ gives
\eqref{eq:minimal-polynomial-B}.  The columns of $M$ show that its image is
spanned by $e_2$ and a vector whose $E_0$ component is
$\eta\rho/L\neq0$; hence $\operatorname{rank}M=2$.  When
$\mathfrak c\neq0$, the minimal polynomial is exactly
$x(x^2-\mathfrak c)$.  If $\mathfrak c=0$, the $(E_0,E_0)$ entry of
$M^2$ is $\eta\rho^2/L^2\neq0$, so the minimal polynomial is exactly
$x^3$.  In particular, the square-zero nilpotent type in
Lemma~\ref{lem:rank-two-normal-forms}(3) is excluded.  Items~1--2 of that
lemma now identify the invariant plane and prove the three normal-form
statements in \eqref{eq:orbit-type-sign}.  When $\mathfrak c<0$, the image plane is
definite and contains $e_2$, so it is positive definite for $\eta=+1$ and
negative definite for $\eta=-1$.  In the parabolic case, the non-null
direction in $\operatorname{im}\mathcal B$ likewise has sign $\eta$.
Thus $\eta$ records the causal subtype not determined by the sign of
$\mathfrak c$ alone.
\end{proof}

\medskip
\noindent\textbf{Geometric role of the invariant.}
The scalar $\mathfrak c$ is more than a first integral useful for solving the
ODEs.  It identifies the conjugacy type of the actual ambient Killing
generator.  Thus a sign computed from the profile data decides whether the
surface is swept out by rotations, boosts, or null rotations, and the
zero-curvature compatibility of the moving frame proves that this generator
is constant along the surface.

\begin{remark}
Theorem \ref{thm:rotational-first-integral} supplies the missing
extrinsic condition in the abstract cohomogeneity-one construction.
The warped metric and Gauss--Codazzi equations alone do not label the
orbit as elliptic, hyperbolic, or parabolic.  The sign of
$\mathfrak c$ gives this broad normal-form classification, while the orbit
sign $\eta$ records the remaining causal subtype in the definite and
parabolic cases.
\end{remark}

\subsection{The weighted equations in symmetry variables}

Let
\begin{equation}
    u=\lambda^2H.
\end{equation}
For a function depending only on $s$,
\begin{equation}
    \Delta_g u
    =
    \sigma(u''+qu').
\end{equation}
The conformal-biharmonic equations therefore reduce to
\begin{align}
    \sigma(u''+qu')
    +
    \left[
        -\varepsilon(k_1^2+k_2^2)
        -\frac{2}{L^2}
    \right]u
    &=
    0,
    \label{eq:rot-normal-CB-analytical}\\
    k_1u'
    +
    \varepsilon uH'
    &=
    0.
    \label{eq:rot-tangent-CB-analytical}
\end{align}
On an interval on which $u\neq0$ and $k_1\neq0$, define
\begin{equation}
    w=\frac{u'}{u}.
\end{equation}
Equation \eqref{eq:rot-tangent-CB-analytical} gives
\begin{equation}
    w
    =
    -\varepsilon\frac{H'}{k_1}.
    \label{eq:w-from-H}
\end{equation}
The normal equation becomes the Riccati equation
\begin{equation}
    \boxed{
    \sigma(w'+w^2+qw)
    -
    \varepsilon(k_1^2+k_2^2)
    -
    \frac{2}{L^2}
    =
    0.
    }
    \label{eq:rot-Riccati}
\end{equation}
Equations
\eqref{eq:orbit-invariant},
\eqref{eq:rot-codazzi-analytical},
\eqref{eq:w-from-H}, and
\eqref{eq:rot-Riccati}
give a purely analytic profile-curve formulation.

If $q\neq0$, Codazzi gives
\begin{equation}
    k_1
    =
    k_2+\frac{k_2'}{q}.
    \label{eq:k1-from-k2}
\end{equation}
The orbit invariant gives
\begin{equation}
    k_2^2
    =
    \varepsilon
    \left(
        \frac{1}{L^2}
        -\sigma q^2
        -\frac{\mathfrak c}{\eta\rho^2}
    \right).
    \label{eq:k2-from-orbit-invariant}
\end{equation}
Thus a choice of orbit type and one sign of $k_2$ reduces the problem
to an ordinary differential equation for the profile radius $\rho$.

\subsection{What elliptic, hyperbolic, and parabolic symmetry impose}

For a spacelike surface,
\begin{equation}
    \sigma=\eta=+1,
    \qquad
    \varepsilon=-1,
\end{equation}
and
\begin{equation}
    \mathfrak c
    =
    \rho^2
    \left(
        \frac{1}{L^2}-q^2+k_2^2
    \right).
\end{equation}
After writing $|\mathfrak c|=\omega^2$, the three cases are
\begin{align}
    \text{elliptic:}\qquad
    k_2^2
    &=
    q^2-\frac{1}{L^2}-\frac{\omega^2}{\rho^2},
    \label{eq:spacelike-elliptic-constraint}\\
    \text{hyperbolic:}\qquad
    k_2^2
    &=
    q^2-\frac{1}{L^2}+\frac{\omega^2}{\rho^2},
    \label{eq:spacelike-hyperbolic-constraint}\\
    \text{parabolic:}\qquad
    k_2^2
    &=
    q^2-\frac{1}{L^2}.
    \label{eq:spacelike-parabolic-constraint}
\end{align}

For a timelike surface with spacelike profile and timelike orbit,
\begin{equation}
    \sigma=+1,
    \qquad
    \eta=-1,
    \qquad
    \varepsilon=+1,
\end{equation}
and
\begin{equation}
    \mathfrak c
    =
    \rho^2
    \left(
        q^2+k_2^2-\frac{1}{L^2}
    \right).
\end{equation}
The corresponding constraints are
\begin{align}
    \text{elliptic:}\qquad
    k_2^2
    &=
    \frac{1}{L^2}-q^2-\frac{\omega^2}{\rho^2},
    \label{eq:timelike-elliptic-constraint}\\
    \text{hyperbolic:}\qquad
    k_2^2
    &=
    \frac{1}{L^2}-q^2+\frac{\omega^2}{\rho^2},
    \label{eq:timelike-hyperbolic-constraint}\\
    \text{parabolic:}\qquad
    k_2^2
    &=
    \frac{1}{L^2}-q^2.
    \label{eq:timelike-parabolic-constraint}
\end{align}

\subsection{The spacelike parabolic branch as a scalar third-order problem}

The generic spacelike parabolic branch is the most directly
tractable analytically.  We fix the profile orientation by working on an interval on which
$q>0$ and later restrict to the open branch $k_1\neq0$ when eliminating the
weighted mean curvature.  Assume
\begin{equation}
    \sigma=\eta=+1,
    \qquad
    \varepsilon=-1,
    \qquad
    \mathfrak c=0.
\end{equation}
Equation
\eqref{eq:spacelike-parabolic-constraint} can be parametrized by a
function $\theta$:
\begin{equation}
    q=\frac{\cosh\theta}{L},
    \qquad
    k_2=\frac{\sinh\theta}{L}.
    \label{eq:parabolic-rapidity}
\end{equation}
This choice fixes the $q>0$ sheet.  A solution on the $q<0$ sheet is carried
to this convention by reversing the profile coordinate
$\widetilde s=-s$ and setting
$\widetilde\theta(\widetilde s)=\theta(-\widetilde s)$; the odd profile
jets change sign.  Thus no local branch is lost by the orientation choice.
The Codazzi equation gives
\begin{equation}
    k_1
    =
    \theta'
    +
    \frac{\sinh\theta}{L}.
    \label{eq:parabolic-k1}
\end{equation}
Indeed,
\begin{equation}
    k_2'
    =
    \frac{\theta'\cosh\theta}{L}
    =
    q(k_1-k_2).
\end{equation}
The mean curvature is
\begin{equation}
    H
    =
    -\frac{\theta'}{2}
    -
    \frac{\sinh\theta}{L}.
    \label{eq:parabolic-H}
\end{equation}
The tangential equation gives
\begin{equation}
    w
    =
    \frac{H'}{k_1}.
    \label{eq:parabolic-w}
\end{equation}
Substitution into the normal equation produces the scalar third-order
equation
\begin{equation}
\boxed{
    \left(\frac{H'}{k_1}\right)'
    +
    \left(\frac{H'}{k_1}\right)^2
    +
    \frac{\cosh\theta}{L}\frac{H'}{k_1}
    +
    k_1^2
    +
    \frac{\sinh^2\theta}{L^2}
    -
    \frac{2}{L^2}
    =
    0,
}
\label{eq:parabolic-third-order-compact}
\end{equation}
where $k_1$ and $H$ are given by
\eqref{eq:parabolic-k1} and \eqref{eq:parabolic-H}.

Equivalently, wherever
\begin{equation}
    L\theta'+\sinh\theta\neq0,
\end{equation}
equation \eqref{eq:parabolic-third-order-compact} is
\begin{align}
    \theta'''
    =
    \frac{1}{
        2L^3\left(L\theta'+\sinh\theta\right)
    }
    \Big[
    &4L^4(\theta')^4
    +3L^4(\theta'')^2
    +12L^3(\theta')^3\sinh\theta
    \nonumber\\
    &+4L^3\theta'\theta''\cosh\theta
    +28L^2(\theta')^2\sinh^2\theta
    -4L^2(\theta')^2
    \nonumber\\
    &-3L^2\theta''\sinh(2\theta)
    +20L\theta'\sinh^3\theta
    -20L\theta'\sinh\theta
    \nonumber\\
    &+8\sinh^4\theta
    -8\sinh^2\theta
    \Big].
    \label{eq:parabolic-third-order-expanded}
\end{align}

Once $\theta$ is known, the remaining quantities are obtained by
quadrature:
\begin{align}
    \rho(s)
    &=
    \rho_0
    \exp
    \left[
        \frac{1}{L}
        \int_0^s\cosh\theta(\tau)\,d\tau
    \right],
    \label{eq:parabolic-rho-quadrature}\\
    u(s)
    &=
    u_0
    \exp
    \left[
        \int_0^s
        \frac{H'(\tau)}{k_1(\tau)}
        \,d\tau
    \right],
    \label{eq:parabolic-u-quadrature}\\
    \lambda(s)^2
    &=
    \frac{u(s)}{H(s)}.
    \label{eq:parabolic-lambda}
\end{align}

\begin{theorem}[Local spacelike parabolic family]
\label{thm:analytical-local-parabolic-family}
Fix initial data
\begin{equation}
    \theta(0)=\theta_0,
    \qquad
    \theta'(0)=a_0,
    \qquad
    \theta''(0)=b_0,
    \qquad
    \rho(0)=\rho_0>0.
\end{equation}
Define
\begin{align}
    k_{1,0}
    &=
    a_0+\frac{\sinh\theta_0}{L},
    \\
    H_0
    &=
    -\frac{a_0}{2}
    -\frac{\sinh\theta_0}{L},
    \\
    H'_0
    &=
    -\frac{b_0}{2}
    -\frac{a_0\cosh\theta_0}{L}.
\end{align}
Assume
\begin{equation}
    k_{1,0}\neq0,
    \qquad
    H_0\neq0,
    \qquad
    H'_0\neq0,
    \qquad
    H_0\neq k_{1,0}.
    \label{eq:parabolic-open-data}
\end{equation}
Choose $u_0$ so that
\begin{equation}
    \frac{u_0}{H_0}>0.
\end{equation}
Then the scalar initial-value problem has a unique local real-analytic
solution $\theta$, and the associated functions $\rho,u,\lambda$ are
uniquely determined by the stated initial data.  Moreover, there exists a
local real-analytic spacelike parabolic rotational immersion
\begin{equation}
    X:(-\delta,\delta)\times J
    \longrightarrow
    \operatorname{AdS}_3(L)
\end{equation}
which is a proper biharmonic conformal immersion.  The immersion is unique
up to an ambient isometry of $\operatorname{AdS}_3(L)$.  Both $H$ and
$\lambda$ are nonconstant.

The conditions \eqref{eq:parabolic-open-data} define an open subset of
the initial-data space.
\end{theorem}

\begin{proof}
Equation \eqref{eq:parabolic-third-order-expanded} is an analytic
ordinary differential equation in a neighborhood of the initial data
because $k_{1,0}\neq0$.  Standard analytic ODE theory therefore gives
a unique local analytic solution $\theta$.

Equations
\eqref{eq:parabolic-rho-quadrature} and
\eqref{eq:parabolic-u-quadrature}
then define positive $\rho$ and nonzero $u$.  The quantities
$q,k_1,k_2,H$ satisfy Gauss, Codazzi, and both
conformal-biharmonic equations by construction.  The fundamental theorem
of hypersurfaces in a semi-Riemannian space form first gives a local
spacelike immersion, unique up to ambient isometry.  Its rotational first
integral is $\mathfrak c=0$; Theorem
\ref{thm:rotational-first-integral} then shows that the immersion is
generated by a constant parabolic element of $\mathfrak{so}(2,2)$.

Since $H'_0\neq0$, the mean curvature is nonconstant.  Moreover,
\begin{equation}
    \frac{(\lambda^2)'}{\lambda^2}
    =
    \frac{u'}{u}-\frac{H'}{H}
    =
    H'
    \left(
        \frac{1}{k_1}-\frac{1}{H}
    \right).
    \label{eq:lambda-log-derivative-parabolic}
\end{equation}
The assumptions $H'_0\neq0$ and $H_0\neq k_{1,0}$ imply
\begin{equation}
    (\lambda^2)'(0)\neq0.
\end{equation}
After shrinking $\delta$ if necessary,
$u/H>0$ throughout the interval.  Hence $\lambda$ is positive and
nonconstant.  Finally, $H_0\neq0$ implies that the immersion is
nonharmonic and therefore proper biharmonic.
\end{proof}

\begin{remark}[Three profile parameters on the generic parabolic branch]
\label{rem:spacelike-parabolic-parameter-count}
For fixed $L$ and a fixed base point $s=0$, the scalar initial-value problem
is determined by
\[
(\theta_0,a_0,b_0)
=
\bigl(\theta(0),\theta'(0),\theta''(0)\bigr),
\]
together with the positive orbit scale $\rho_0$ and the nonzero weight
scale $u_0$.  As in Proposition~\ref{prop:cohom1-parameter-count},
$\rho_0$ is removed by positive rescaling of the orbit coordinate and
$u_0$ by constant homothety of the source metric.  Thus, modulo these
gauges and ambient isometry, the generic pointed spacelike parabolic germs
form a three-dimensional continuous parameter family on the free generic
stratum.  Profile, orbit, and normal reversals give residual discrete
identifications, and exceptional data may have additional stabilizers; no
global smooth-moduli-space assertion is intended.
\end{remark}

\begin{example}[Concrete analytic initial data]
\label{ex:concrete-parabolic-data}
Set $L=1$ and choose
\begin{equation}
    \theta_0=0,
    \qquad
    a_0=1,
    \qquad
    b_0=0,
    \qquad
    \rho_0=1,
    \qquad
    u_0=-1.
\end{equation}
Then
\begin{equation}
    q_0=1,
    \qquad
    k_{2,0}=0,
    \qquad
    k_{1,0}=1,
    \qquad
    H_0=-\frac12,
    \qquad
    H'_0=-1.
\end{equation}
Furthermore,
\begin{equation}
    \lambda_0^2=\frac{u_0}{H_0}=2,
\end{equation}
and
\begin{equation}
    (\lambda^2)'(0)=-6.
\end{equation}
Thus these data determine a rigorous local spacelike parabolic
rotational example with nonconstant mean curvature and nonconstant
dilation.
\end{example}

\subsection{Reconstructing the spacelike surface in null coordinates}
\label{subsec:canonical-parabolic-immersion}

The preceding existence theorem can be realized by an explicit parabolic
parametrization.  Introduce null coordinates $(U,V,Y,Z)$ on
$\mathbb R^{2,2}$ with
\begin{equation}
    \langle dX,dX\rangle_{2,2}
    =
    -2\,dU\,dV-dY^2+dZ^2.
    \label{eq:null-ambient-metric}
\end{equation}
The maps
\begin{equation}
    \mathcal P_t(U,V,Y,Z)
    =
    \left(
        U,\,
        V+tZ+\frac12t^2U,\,
        Y,\,
        Z+tU
    \right)
    \label{eq:parabolic-action}
\end{equation}
form a one-parameter subgroup of $O(2,2)$.  Its infinitesimal generator
$\mathcal B_{\mathrm{par}}$ satisfies
\begin{equation}
    \mathcal B_{\mathrm{par}}^3=0,
    \qquad
    \mathcal B_{\mathrm{par}}^2\neq0.
\end{equation}

Let $\theta$ solve \eqref{eq:parabolic-third-order-compact}.  Define
$\rho$ and $Y$ by
\begin{align}
    \rho'
    &=
    \frac{\cosh\theta}{L}\rho,
    \qquad
    \rho(0)=\rho_0>0,
    \label{eq:canonical-rho-equation}\\
    Y'
    &=
    \frac{\cosh\theta}{L}Y+\sinh\theta,
    \qquad
    Y(0)=Y_0,
    \label{eq:canonical-Y-equation}
\end{align}
and set
\begin{equation}
    V=\frac{L^2-Y^2}{2\rho}.
    \label{eq:canonical-V}
\end{equation}
Then
\begin{equation}
    \boxed{
    X(s,t)
    =
    \mathcal P_t\bigl(\rho(s),V(s),Y(s),0\bigr)
    =
    \left(
        \rho,\,
        V+\frac12t^2\rho,\,
        Y,\,
        t\rho
    \right)
    }
    \label{eq:canonical-parabolic-immersion}
\end{equation}
takes values in $\AdS_3(L)$.

\begin{proposition}[Explicit spacelike reconstruction]
\label{prop:explicit-parabolic-reconstruction}
The immersion \eqref{eq:canonical-parabolic-immersion} is spacelike and has
induced metric
\begin{equation}
    g=ds^2+\rho(s)^2dt^2.
    \label{eq:canonical-parabolic-metric}
\end{equation}
With a suitable timelike unit normal, its principal curvatures are
\begin{equation}
    k_1
    =
    \theta'+\frac{\sinh\theta}{L},
    \qquad
    k_2
    =
    \frac{\sinh\theta}{L}.
    \label{eq:canonical-parabolic-curvatures}
\end{equation}
Consequently, whenever $\theta$, $u$, and $\lambda$ are obtained from
Theorem~\ref{thm:analytical-local-parabolic-family}, the map
\begin{equation}
    X:
    \left(I\times J,\lambda^{-2}
    \left(ds^2+\rho^2dt^2\right)\right)
    \longrightarrow\AdS_3(L)
\end{equation}
is a proper biharmonic conformal immersion with parabolic rotational
symmetry.
\end{proposition}

\begin{proof}
Equation \eqref{eq:canonical-V} gives
\begin{equation}
    -2\rho V-Y^2=-L^2,
\end{equation}
and \eqref{eq:parabolic-action} preserves
\eqref{eq:null-ambient-metric}; hence $X(s,t)\in\AdS_3(L)$.
At $t=0$,
\begin{equation}
    X_t=(0,0,0,\rho),
\end{equation}
so $\langle X_t,X_t\rangle=\rho^2$.  Moreover,
$\langle X_s,X_t\rangle=0$.  Differentiating
\eqref{eq:canonical-V} gives
\begin{equation}
    \langle X_s,X_s\rangle
    =
    L^2q^2-(Y'-qY)^2,
    \qquad
    q=\frac{\rho'}{\rho}.
\end{equation}
Equations \eqref{eq:canonical-rho-equation} and
\eqref{eq:canonical-Y-equation} therefore imply
\begin{equation}
    \langle X_s,X_s\rangle
    =
    \cosh^2\theta-\sinh^2\theta
    =
    1.
\end{equation}
This proves \eqref{eq:canonical-parabolic-metric}.

For completeness, along $t=0$ one may choose
\begin{equation}
\begin{split}
    \xi_0
    =
    \Bigg(
    &-\frac{\rho\sinh\theta}{L},\\
    &\frac{
        L\bigl(L\sinh\theta+2Y\cosh\theta\bigr)
        +Y^2\sinh\theta
    }{2L\rho},\\
    &-\cosh\theta-\frac{Y\sinh\theta}{L},
    \,0
    \Bigg).
\end{split}
\label{eq:canonical-parabolic-normal}
\end{equation}
A direct calculation gives
\begin{equation}
    \langle\xi_0,\xi_0\rangle=-1,
    \qquad
    \langle\xi_0,X\rangle
    =
    \langle\xi_0,X_s\rangle
    =
    \langle\xi_0,X_t\rangle
    =
    0.
\end{equation}
Extending the normal by
\begin{equation}
    \xi(s,t)=\mathcal P_t\xi_0(s)
\end{equation}
and evaluating the second fundamental form yields
\begin{equation}
    \langle X_{st},\xi\rangle=0,
    \qquad
    \langle X_{tt},\xi\rangle=\rho^2\frac{\sinh\theta}{L},
    \qquad
    \langle X_{ss},\xi\rangle
    =
    \theta'+\frac{\sinh\theta}{L}.
\end{equation}
Thus the coordinate directions are principal and
\eqref{eq:canonical-parabolic-curvatures} follows.  The final assertion is
then exactly Theorem~\ref{thm:analytical-local-parabolic-family}.
\end{proof}

The auxiliary profile function in \eqref{eq:canonical-Y-equation} is also
given explicitly by
\begin{equation}
    Y(s)
    =
    \rho(s)
    \left[
        \frac{Y_0}{\rho_0}
        +
        \int_0^s
        \frac{\sinh\theta(\tau)}{\rho(\tau)}
        \,d\tau
    \right].
    \label{eq:canonical-Y-quadrature}
\end{equation}
Thus the entire parabolic immersion, including its conformal dilation, is
recovered from the scalar function $\theta$ by quadratures.

\subsection{The rigid umbilical branch as a control solution}

There is also a closed-form constant-dilation solution that provides a
useful analytic check.  Let $\theta=\theta_\ast$ be constant.  Then
\begin{equation}
    k_1=k_2=\frac{\sinh\theta_\ast}{L},
    \qquad
    q=\frac{\cosh\theta_\ast}{L}.
\end{equation}
The conformal-biharmonic equations with $H\neq0$ require
\begin{equation}
    \sinh^2\theta_\ast=1.
\end{equation}
Hence
\begin{equation}
    k_1=k_2=\pm\frac1L,
    \qquad
    q=\frac{\sqrt2}{L},
\end{equation}
and
\begin{equation}
    \rho(s)
    =
    \rho_0e^{\sqrt2\,s/L}.
\end{equation}
The tangential equation forces $u$ to be constant, so
$\lambda^2=u/H$ is constant.  This is the parabolic-coordinate
realization of the totally umbilical proper biharmonic control branch;
it is consistent with the spacelike CMC rigidity theorem.

For a direct geometric visualization of the reconstruction, set $Y_0=0$
and consider the meridian generating curve
\begin{equation}
    \gamma(s)
    =
    X(s,0)
    =
    \left(
        \rho(s),
        \frac{L^2-Y(s)^2}{2\rho(s)},
        Y(s),
        0
    \right).
    \label{eq:meridian-generating-curve}
\end{equation}
Figure~\ref{fig:parabolic-generating-curves} compares the non-CMC curve from
Example~\ref{ex:concrete-parabolic-data} with the closed-form totally
umbilical control branch.  The two curves use $L=1$, $\rho_0=1$, and
$Y_0=0$, and therefore pass through the same point
$\gamma(0)=(1,1/2,0,0)$.  Since $U=\rho>0$ and
$V=(1-Y^2)/(2U)$ on this section, the pair $(U,Y)$ determines the complete
curve $\gamma$; the figure is therefore a coordinate representation of the
actual generating curves rather than a plot of scalar diagnostics.  The
displayed interval is chosen only for visualization and the figure is not
used in any proof.

\begin{figure}[H]
\centering
\includegraphics[width=0.76\textwidth]{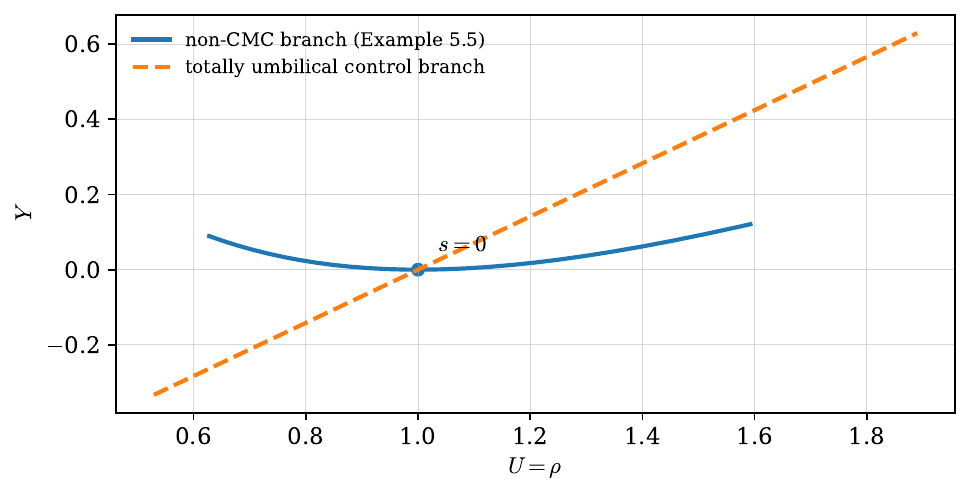}
\caption{Meridian generating curves
$\gamma(s)=X(s,0)=(U,V,Y,0)$ in the $Z=0$ null-coordinate section, represented
in the $(U,Y)$ chart for $-0.45\le s\le0.45$.  The solid curve is the
non-CMC branch of Example~\ref{ex:concrete-parabolic-data}; the dashed curve
is the totally umbilical constant-dilation control branch with
$\sinh\theta_\ast=1$.  In the region $U>0$, the omitted coordinate is
recovered uniquely from $V=(1-Y^2)/(2U)$, and both curves satisfy
$-2UV-Y^2=-1$.}
\label{fig:parabolic-generating-curves}
\end{figure}

\subsection{The real-principal timelike parabolic counterpart}
\label{subsec:timelike-parabolic}

We now complete the local parabolic analysis for a timelike surface with
spacelike profile and timelike orbit.  Throughout this subsection we restrict
to the real-principal branch for which the shape operator is diagonal in the
chosen pseudo-orthonormal frame.  The complex-principal and Jordan/null-
principal types, together with the opposite causal allocation of profile and
orbit directions, are not treated here.  Thus
\begin{equation}
    \sigma=+1,
    \qquad
    \eta=-1,
    \qquad
    \varepsilon=+1.
\end{equation}
The parabolic constraint is
\begin{equation}
    q^2+k_2^2=\frac1{L^2}.
    \label{eq:timelike-parabolic-constraint-zero}
\end{equation}
On an interval on which this parametrization is valid, write
\begin{equation}
    q=\frac{\cos\theta}{L},
    \qquad
    k_2=\frac{\sin\theta}{L}.
    \label{eq:timelike-parabolic-angle}
\end{equation}
Then the Gauss--Codazzi equations give
\begin{equation}
    k_1
    =
    \theta'
    +
    \frac{\sin\theta}{L}.
    \label{eq:timelike-parabolic-k1}
\end{equation}
Indeed, Codazzi gives this identity wherever $\cos\theta\neq0$, while
Gauss gives it wherever $\sin\theta\neq0$; by continuity the identity
holds throughout the interval.  Moreover,
\begin{equation}
    H
    =
    \frac{\theta'}{2}
    +
    \frac{\sin\theta}{L}.
    \label{eq:timelike-parabolic-H}
\end{equation}
On a branch with $u\neq0$ and $k_1\neq0$, the tangential equation yields
\begin{equation}
    w:=\frac{u'}{u}
    =
    -\frac{H'}{k_1},
    \label{eq:timelike-parabolic-w}
\end{equation}
while the normal equation is
\begin{equation}
    \boxed{
    w'+w^2+\frac{\cos\theta}{L}w
    -
    k_1^2
    -
    \frac{\sin^2\theta}{L^2}
    -
    \frac{2}{L^2}
    =
    0.
    }
    \label{eq:timelike-parabolic-scalar}
\end{equation}
Equations
\eqref{eq:timelike-parabolic-k1}--\eqref{eq:timelike-parabolic-scalar}
form a third-order scalar equation for $\theta$.  More explicitly,
\[
H'
=
\frac{\theta''}{2}
+
\frac{\theta'\cos\theta}{L},
\]
so the coefficient of $\theta'''$ in $w'$ is
$-1/(2k_1)$.  Hence, wherever $k_1\neq0$,
\eqref{eq:timelike-parabolic-scalar} can be solved uniquely in the analytic
normal form
\begin{equation}
    \theta'''
    =
    \mathcal F_{\mathrm{tim}}
    (\theta,\theta',\theta''),
    \label{eq:timelike-third-order-normal-form}
\end{equation}
for a real-analytic function $\mathcal F_{\mathrm{tim}}$ on the open set
$L\theta'+\sin\theta\neq0$.

Once $\theta$ is known, set
\begin{align}
    \rho(s)
    &=
    \rho_0
    \exp\left[
        \frac1L\int_0^s\cos\theta(\tau)\,d\tau
    \right],
    \label{eq:timelike-rho-quadrature}\\
    u(s)
    &=
    u_0
    \exp\left[
        -\int_0^s\frac{H'(\tau)}{k_1(\tau)}\,d\tau
    \right],
    \label{eq:timelike-u-quadrature}\\
    \lambda(s)^2
    &=
    \frac{u(s)}{H(s)}.
    \label{eq:timelike-lambda-quadrature}
\end{align}

\begin{theorem}[Local real-principal timelike parabolic family]
\label{thm:timelike-local-parabolic-family}
Fix
\begin{equation}
    \theta(0)=\theta_0,
    \qquad
    \theta'(0)=a_0,
    \qquad
    \theta''(0)=b_0,
    \qquad
    \rho(0)=\rho_0>0.
\end{equation}
Define
\begin{align}
    k_{1,0}
    &=
    a_0+\frac{\sin\theta_0}{L},
    \\
    H_0
    &=
    \frac{a_0}{2}+\frac{\sin\theta_0}{L},
    \\
    H'_0
    &=
    \frac{b_0}{2}
    +
    \frac{a_0\cos\theta_0}{L}.
\end{align}
Assume
\begin{equation}
    k_{1,0}\neq0,
    \qquad
    H_0\neq0,
    \qquad
    H'_0\neq0,
    \qquad
    H_0\neq-k_{1,0}.
    \label{eq:timelike-parabolic-open-data}
\end{equation}
Choose $u_0$ so that $u_0/H_0>0$.  Then
\eqref{eq:timelike-parabolic-scalar} has a unique local real-analytic
solution, and the quadratures
\eqref{eq:timelike-rho-quadrature}--\eqref{eq:timelike-lambda-quadrature}
determine a local real-analytic timelike parabolic rotational immersion
\begin{equation}
    X:(-\delta,\delta)\times J\longrightarrow\AdS_3(L)
\end{equation}
such that
\begin{equation}
    X:
    \left(
        (-\delta,\delta)\times J,\,
        \lambda^{-2}
        \left(ds^2-\rho^2dt^2\right)
    \right)
    \longrightarrow\AdS_3(L)
\end{equation}
is proper biharmonic.  Both $H$ and $\lambda$ are nonconstant.  The
conditions \eqref{eq:timelike-parabolic-open-data} define an open subset of
the initial-data space.
\end{theorem}

\begin{proof}
Because $k_{1,0}\neq0$, the scalar equation is equivalent near the initial
point to the analytic normal form
\eqref{eq:timelike-third-order-normal-form}.  Standard analytic ODE theory
therefore gives a unique local analytic solution $\theta$.  The functions
$\rho$ and $u$ are then nonzero on a sufficiently small interval, and the
quantities $q,k_1,k_2,H$ satisfy Gauss, Codazzi, and the timelike
conformal-biharmonic equations by construction.  The fundamental theorem of
hypersurfaces in a semi-Riemannian space form gives a local timelike
immersion, unique up to ambient isometry.  Its orbit invariant is
$\mathfrak c=0$, so Theorem~\ref{thm:rotational-first-integral} identifies
the symmetry as a parabolic one-parameter action with timelike orbit.

Since $H'_0\neq0$, the mean curvature is nonconstant.  Furthermore,
\begin{equation}
    \frac{(\lambda^2)'}{\lambda^2}
    =
    -H'
    \left(
        \frac1{k_1}+\frac1H
    \right).
    \label{eq:timelike-lambda-log-derivative}
\end{equation}
The last inequality in
\eqref{eq:timelike-parabolic-open-data} therefore implies
$(\lambda^2)'(0)\neq0$.  After shrinking the interval, $u/H>0$, so
$\lambda$ is positive and nonconstant.  Finally, $H_0\neq0$ and
$u_0\neq0$ imply
$\tau_{\lambda^{-2}g}(X)=2u\xi\neq0$, proving proper biharmonicity.
\end{proof}

\begin{example}[Concrete timelike parabolic data]
\label{ex:timelike-concrete-parabolic-data}
Set $L=1$ and choose
\begin{equation}
    \theta_0=0,
    \qquad
    a_0=1,
    \qquad
    b_0=0,
    \qquad
    \rho_0=1,
    \qquad
    u_0=1.
\end{equation}
Then
\begin{equation}
    q_0=1,
    \qquad
    k_{2,0}=0,
    \qquad
    k_{1,0}=1,
    \qquad
    H_0=\frac12,
    \qquad
    H'_0=1.
\end{equation}
Moreover,
\begin{equation}
    \lambda_0^2=2,
    \qquad
    (\lambda^2)'(0)=-6.
\end{equation}
Thus these data lie in the open set
\eqref{eq:timelike-parabolic-open-data} and define a rigorous local
timelike proper biharmonic conformal immersion with both nonconstant mean
curvature and nonconstant dilation.
\end{example}

The timelike family also admits a canonical ambient reconstruction.  In the
null coordinates of \eqref{eq:null-ambient-metric}, define
\begin{equation}
    \widetilde{\mathcal P}_t(U,V,Y,Z)
    =
    \left(
        U,\,
        V-tY-\frac12t^2U,\,
        Y+tU,\,
        Z
    \right).
    \label{eq:timelike-parabolic-action}
\end{equation}
This is the one-parameter subgroup generated by
$B_{\mathrm{par}}^{(-)}$ in
\eqref{eq:parabolic-normal-form}.  Let $\theta$ solve
\eqref{eq:timelike-parabolic-scalar}, define $\rho$ by
\eqref{eq:timelike-rho-quadrature}, and let $Z$ solve
\begin{equation}
    Z'
    =
    \frac{\cos\theta}{L}Z+\sin\theta,
    \qquad
    Z(0)=Z_0.
    \label{eq:timelike-Z-equation}
\end{equation}
Set
\begin{equation}
    V
    =
    \frac{L^2+Z^2}{2\rho}.
    \label{eq:timelike-canonical-V}
\end{equation}
Then
\begin{equation}
    \boxed{
    X(s,t)
    =
    \widetilde{\mathcal P}_t
    \bigl(\rho(s),V(s),0,Z(s)\bigr)
    =
    \left(
        \rho,\,
        V-\frac12t^2\rho,\,
        t\rho,\,
        Z
    \right).
    }
    \label{eq:timelike-canonical-immersion}
\end{equation}

\begin{proposition}[Explicit timelike reconstruction]
\label{prop:explicit-timelike-parabolic-reconstruction}
The map \eqref{eq:timelike-canonical-immersion} takes values in
$\AdS_3(L)$ and has induced metric
\begin{equation}
    g=ds^2-\rho(s)^2dt^2.
    \label{eq:timelike-canonical-metric}
\end{equation}
With a suitable spacelike unit normal, its principal curvatures are
\begin{equation}
    k_1
    =
    \theta'+\frac{\sin\theta}{L},
    \qquad
    k_2
    =
    \frac{\sin\theta}{L}.
    \label{eq:timelike-canonical-curvatures}
\end{equation}
Consequently, the data of
Theorem~\ref{thm:timelike-local-parabolic-family} reconstruct explicitly as
proper biharmonic conformal immersions with timelike parabolic symmetry.
\end{proposition}

\begin{proof}
Equation \eqref{eq:timelike-canonical-V} gives
\[
-2\rho V+Z^2=-L^2,
\]
and \eqref{eq:timelike-parabolic-action} preserves
\eqref{eq:null-ambient-metric}; hence
$X(s,t)\in\AdS_3(L)$.  At $t=0$,
\[
X_t=(0,0,\rho,0),
\]
so $\langle X_t,X_t\rangle=-\rho^2$ and
$\langle X_s,X_t\rangle=0$.  Writing $q=\rho'/\rho$, differentiation of
\eqref{eq:timelike-canonical-V} gives
\begin{equation}
    \langle X_s,X_s\rangle
    =
    L^2q^2+(Z'-qZ)^2.
\end{equation}
Equations \eqref{eq:timelike-parabolic-angle} and
\eqref{eq:timelike-Z-equation} therefore yield
\[
\langle X_s,X_s\rangle
=
\cos^2\theta+\sin^2\theta
=
1,
\]
which proves \eqref{eq:timelike-canonical-metric}.

Along $t=0$, define
\begin{equation}
\begin{split}
    \xi_0
    =
    \Bigg(
    &-\frac{\rho\sin\theta}{L},\\
    &\frac{
        L^2\sin\theta
        +2LZ\cos\theta
        -Z^2\sin\theta
    }{2L\rho},\\
    &0,\,
    \cos\theta-\frac{Z\sin\theta}{L}
    \Bigg).
\end{split}
\label{eq:timelike-canonical-normal}
\end{equation}
A direct calculation gives
\begin{equation}
    \langle\xi_0,\xi_0\rangle=1,
    \qquad
    \langle\xi_0,X\rangle
    =
    \langle\xi_0,X_s\rangle
    =
    \langle\xi_0,X_t\rangle
    =
    0.
\end{equation}
Extend the normal by
\[
\xi(s,t)=\widetilde{\mathcal P}_t\xi_0(s).
\]
Using \eqref{eq:timelike-Z-equation} and
\eqref{eq:timelike-parabolic-angle}, one finds
\begin{equation}
    \langle X_{st},\xi\rangle=0,
    \qquad
    \langle X_{tt},\xi\rangle
    =
    -\rho^2\frac{\sin\theta}{L},
    \qquad
    \langle X_{ss},\xi\rangle
    =
    \theta'+\frac{\sin\theta}{L}.
\end{equation}
Since the normal sign is $\varepsilon=+1$ and
$g_{tt}=-\rho^2$, these identities give precisely
\eqref{eq:timelike-canonical-curvatures}.  The final assertion follows from
Theorem~\ref{thm:timelike-local-parabolic-family}.
\end{proof}

As in the spacelike case, the auxiliary profile coordinate is recovered by
quadrature:
\begin{equation}
    Z(s)
    =
    \rho(s)
    \left[
        \frac{Z_0}{\rho_0}
        +
        \int_0^s
        \frac{\sin\theta(\tau)}{\rho(\tau)}
        \,d\tau
    \right].
    \label{eq:timelike-Z-quadrature}
\end{equation}
Thus, after quotienting by positive orbit-coordinate rescaling and constant
source homothety, the generic pointed real-principal timelike parabolic germ
has three continuous profile parameters $(\theta_0,a_0,b_0)$ on the free
generic stratum.  As in the spacelike branch, residual discrete
identifications and exceptional stabilizers are not being quotiented into a
global smooth moduli space.

\section{What the results establish and what remains open}
\label{sec:conclusion}

The paper began with a single geometric question: can a fixed surface in
$\AdS_3$ be made proper biharmonic by changing only the conformal metric on
its domain?  The answer is neither a general rigidity theorem nor an
unrestricted existence theorem.  Instead, the analysis reveals a boundary
between two regimes.

On the spacelike nonminimal CMC branch, the conformal freedom is illusory.
The tangential biharmonic equation forces the weighted mean curvature to be
constant along the kernel directions of the shape operator, and the
Gauss--Codazzi equations then rule out a variable dilation.  The surviving
surface is locally totally umbilical with
$A=\pm L^{-1}\operatorname{id}$ and $K=-2L^{-2}$.  In other words, conformal
reparametrization does not create a new CMC family; it only rescales the
source metric of the known isometric branch.

Outside CMC geometry, the situation changes decisively.  The analytic
cohomogeneity-one system admits an open set of local data for which both the
mean curvature and the dilation vary.  This proves that the CMC theorem is a
sharp branch rigidity statement rather than evidence for a general
nonexistence principle.  After the natural gauges are removed, the generic
pointed cohomogeneity-one profile carries four continuous parameters, while
the parabolic orbit constraint lowers the corresponding profile data to
three.

The moving-frame invariant completes the passage from intrinsic data to
extrinsic geometry.  It produces a constant rank-two generator in
$\mathfrak{so}(2,2)$ and classifies the one-parameter orbit as elliptic,
hyperbolic, or index-three parabolic.  On the parabolic branch, the surface
system reduces to one third-order scalar equation, and the null-coordinate
formulas reconstruct the immersion and its conformal dilation by
quadratures.  Thus the main progression of the paper is complete at the
local level:
\[
\text{rigid CMC branch}
\quad\longrightarrow\quad
\text{open non-CMC profiles}
\quad\longrightarrow\quad
\text{explicit ambient surfaces}.
\]

The remaining problems are global.  The most immediate question is whether
one of the nonconstant-dilation parabolic germs extends to a complete
surface while keeping $H$ away from zero.  For the elliptic and hyperbolic
classes, the existence of periodic profile curves would produce natural
global examples.  A related issue is whether the induced metric and the
conformally rescaled source metric can be complete simultaneously.  On the
timelike side, the present paper treats one real-principal causal branch;
the complex-principal, Jordan/null-principal, and opposite causal
allocations remain to be analyzed.

A useful next theorem would therefore be a continuation criterion for the
scalar third-order flows, separating extension from blow-up in geometric
quantities such as $H$, $\lambda$, and the principal curvatures.  Such a
criterion would turn the local construction developed here into a systematic
route toward complete or periodic biharmonic conformal surfaces.

\section*{Data and code availability}
No external data were used.  The source package accompanying this manuscript
contains the self-contained scripts
\texttt{verify\_bci\_analytics.py} and
\texttt{generate\_parabolic\_generating\_curves.py}.  From the source
directory, the commands
\begin{verbatim}
python verify_bci_analytics.py \
  --manuscript bci_ads3_human_readable_revision.tex
python generate_parabolic_generating_curves.py
\end{verbatim}
write the machine-readable JSON audit and regenerate the PDF figure listed
above.  The JSON output records the software environment and SHA-256 hashes
of the checked files.

\section*{Generative AI disclosure}
OpenAI ChatGPT models were used as auxiliary tools for exploratory algebra,
code drafting, literature organization, consistency checks, and language
editing.  All theorem statements, proofs, computations, citations, and
interpretive claims were checked by the author, who takes full responsibility
for the manuscript.

\appendix
\section{Reproducibility audit}
\label{app:audit}

The analytic results do not depend on numerical computation.  The supplied
script separates exact symbolic identities from floating-point diagnostics.
The reported run used Python~3.13.5, SymPy~1.14.0, NumPy~2.3.5, and
SciPy~1.17.0.  The ordinary differential equations were integrated with
\texttt{solve\_ivp(method="DOP853")}, relative tolerance
$2\times10^{-13}$ and absolute tolerance $2\times10^{-15}$.

The following checks are exact symbolic identities in the script:
\begin{enumerate}
\item the compact spacelike parabolic equation
\eqref{eq:parabolic-third-order-compact} solves for $\theta'''$ as exactly
\eqref{eq:parabolic-third-order-expanded};
\item the orbit matrix is skew-adjoint and satisfies
$M^3=\mathfrak c M$;
\item both displayed index-three parabolic actions preserve
$-2\,dU\,dV-dY^2+dZ^2$;
\item the matrix \eqref{eq:square-zero-nilpotent} is a rank-two,
skew-adjoint, square-zero counterexample to any unrestricted claim that the
index-three blocks exhaust all nilpotent rank-two elements of
$\mathfrak{so}(2,2)$.
\end{enumerate}

The numerical diagnostics are deliberately reported separately.  For the
six-variable example of Section~\ref{sec:cohomogeneity-one-existence},
quintic interpolating splines are differentiated and compared with the
independent Gauss, Codazzi, and normal ODE expressions away from eight points
at each endpoint.  For the explicit
spacelike parabolic reconstruction, the anti-de Sitter constraint, metric,
normal, and second fundamental form are evaluated from the reconstructed
ambient coordinates.  The larger second-form error reflects numerical
second differentiation and is not used as theorem-level evidence.

\begin{table}[htbp]
\centering
\begin{tabular}{lc}
\toprule
Numerical diagnostic & Maximum absolute error \\
\midrule
Cohomogeneity-one Gauss spline residual & $2.36\times10^{-11}$\\
Cohomogeneity-one Codazzi spline residual & $2.25\times10^{-11}$\\
Cohomogeneity-one normal-ODE spline residual & $2.06\times10^{-11}$\\
Canonical parabolic anti-de Sitter constraint & $4.45\times10^{-16}$\\
Canonical parabolic metric & $7.17\times10^{-11}$\\
Canonical parabolic normal relations & $8.96\times10^{-13}$\\
Canonical parabolic second fundamental form & $3.40\times10^{-8}$\\
\bottomrule
\end{tabular}
\caption{Independent floating-point diagnostics for the local examples with
$L=1$.  Exact algebraic identities are listed separately in the text.}
\label{tab:numerical-consistency}
\end{table}

For Example~\ref{ex:concrete-parabolic-data}, the script recovers
\[
H(0)=-\frac12,
\qquad
H'(0)=-1,
\qquad
\lambda^2(0)=2,
\qquad
(\lambda^2)'(0)=-6.
\]
A perturbation test with NumPy seed $20260705$ independently perturbs
$\theta(0)$, $\theta'(0)-1$, and $\theta''(0)$ uniformly in
$[-0.02,0.02]$.  All $50$ trials remain on the open branch
$k_1H\neq0$ over $[-0.025,0.025]$; the smallest observed
$\min(|k_1|,|H|)$ margin is $0.450$.  This experiment illustrates, but does
not prove, the openness established analytically in
Theorem~\ref{thm:analytical-local-parabolic-family}.  Full-precision values,
file hashes, and environment metadata are stored in
\texttt{audit\_results.json}.

\end{document}